\documentclass[reqno,12pt]{amsart}
\usepackage{color}
\usepackage{amsmath}
\usepackage{amsfonts}
\usepackage{amssymb}
\usepackage{amsthm}
\usepackage{newlfont}
\usepackage{graphicx}

\usepackage[margin=1in]{geometry}

\usepackage{tikz}

\newtheorem{thm}{\bfseries Theorem}[section]
\newtheorem{cor}[thm]{\bfseries Corollary}
\newtheorem{rem}[thm]{\bfseries Remark}
\newtheorem{prop}[thm]{\bfseries Proposition}
\newtheorem{lem}[thm]{\bfseries Lemma}

\newtheorem*{rem*}{\bfseries Remark}

\numberwithin{equation}{section}

\newcommand{\ed}{\end {document}}

\newcounter{smalllist}

\renewcommand*{\tilde}{\widetilde}

\title[Linearization]{On the gap property of a linearized NLS operator}
\author[Dong Li]{Dong Li}
\address[Dong Li]{Department of Mathematics, the  University
of Hong Kong, Hong Kong, China}%
\email{mathdl@hku.hk}
\author[Kai Yang]{Kai Yang}
\address[Kai Yang]
{School of Mathematics, Southeast University, Nanjing, Jiangsu Province, 211189, China}
\email{kaiyang@seu.edu.cn}

\begin{document}

\begin{abstract}
We consider a pair of linear operators corresponding to the linearization around the ground state soliton of the cubic nonlinear Schr\"odinger equation
in dimension three. We introduce a new {comparison-based} approach and rigorously prove that the interval $(0, 1]$ does not contain any eigenvalues  of these operators.  Furthermore we show the absence of resonances at the bottom of the essential spectrum. All the obtained results are for the fully non-radial case. The method can be adapted to many other
spectral problems.
\end{abstract}
\maketitle
\section{Introduction}

We consider the nonlinear Schr\"odinger equation for
$\psi = \psi(t,x): \mathbb R \times \mathbb R^3 \to \mathbb C$:
\begin{align}
i \partial_t \psi + \Delta \psi + |\psi|^2 \psi=0.
\end{align}
Plugging in the standing wave ansatz $\psi = e^{it} \phi(x)$, we obtain
\begin{align}
-\Delta \phi + \phi - |\phi|^2 \phi=0. 
\end{align}
Denote by $Q$ the {unique} positive radial ground state {(cf. \cite{C, K})}. We have $Q(x) = Q(r)$, ($r=|x|$),\footnote{We shall slightly abuse
the notation and regard $Q(x)=Q(|x|)=Q(r)$ when there is no obvious confusion.} where $Q$
solves the nonlinear ODE 
\begin{align}\label{ode}
-Q^{\prime\prime}(r) - \frac 2 r Q^{\prime}(r) + Q(r) - Q^3(r) =0.
\end{align}
Consider $\phi = Q + \eta$ with $\eta=\eta_1 + i \eta_2$. Clearly
\begin{align}
 & \Delta \phi - \phi + |\phi|^2 \phi \notag \\
 =& \Delta \eta - \eta + (Q^2 + 2 Q \eta_1) ( Q + \eta_1+ i\eta_2)  - Q^3 +\mathcal O(|\eta|^2) 
 \notag \\
 =&\; L_+ \eta_1 +i L_- \eta_2 + \mathcal O(|\eta|^2),
 \end{align}
 where $L_+ = -\Delta +1 -3Q^2$, $L_-= -\Delta+1 -Q^2$. 

It is known that the essential spectrum of $L_+$ and $L_-$ is $[1, \infty)$.  $L_+$ has
a unique negative bound state. If $f \perp \Delta Q$, then  (below $\langle\cdot, \cdot\rangle$
denotes the usual $L^2$-inner product for real-valued functions)
\begin{align}
\langle L_+ f , f \rangle \gtrsim   (\int_{\mathbb R^3} f Q dx )^2.
\end{align}
The kernel of $L_+$ is $\textrm{span}\{ \partial_j Q \}_{j=1}^3$ and the kernel of $L_-$ is
$\textrm{span}\{Q\}$ (cf. \cite{Wein}). On the other hand, it has long been accepted wisdom (cf. \cite{F74, P01, DS, S2}) that $L_+$ and $L_-$ have no eigenvalues in $(0,1]$ and that $1$ is not a resonance of either operator--a property known as the gap property.
 This gap property plays an important role in the construction of stable manifolds
for orbitally unstable NLS (cf. \cite{NS} and \cite{S2}). It was numerically verified by Demanet and Schlag
in \cite{DS}  using the Birman-Schwinger method for NLS with nonlinearities $|\psi|^{2\beta}
\psi$, $\beta_*<\beta\le 1$, $\beta_*\approx 0.913958905$.  

In their recent work \cite{CHS}, Costin, Huang, and Schlag gave a rigorous proof of the gap property under the assumption of radial symmetry. The main contributions of \cite{CHS} are twofold.

First, they constructed a remarkably accurate approximate ground state $\tilde Q$, which differs from the true ground state by less than $10^{-4}$. More precisely, the pointwise error is bounded in absolute value by $7\cdot 10^{-5} \cdot \frac{1}{1+r} e^{-r}$.

Second, they developed a robust Wronskian-based approach that connects two Jost-type quasi-solutions: one starting from $r=0$ and the other from $r=\infty$ (the decaying solution at infinity).

The key step is verifying that 
\[
\inf_{\lambda \in [0,1]} |W(\lambda)| > 0 \quad \text{for $L_+$}, 
\qquad 
\inf_{\lambda \in [0,1]} |W(\lambda)/\lambda| > 0 \quad \text{for $L_-$},
\]
where $\lambda$ is the spectral parameter. This highly nontrivial computation was carried out in \cite{CHS} to establish the gap property in the radial setting.

The purpose of this paper is to give a rigorous proof of this gap property for the fully non-radial case.
We shall develop a new \textsf{comparison-based} approach which offers an interesting
(and perhaps simpler)
alternative to the Wronskian strategy developed in \cite{CHS}.  

\begin{thm} \label{thMain}
Consider
\begin{align}
L_+=-\Delta +1 -3Q^2, \qquad L_-=-\Delta +1-Q^2,
\end{align}
where $Q$ is the unique positive radial ground state solving  \eqref{ode}.

The operators $L_+$ and $L_-$ do not have any  eigenvalues in $(0,1]$. 
For the eigenvalue $\lambda=0$, the kernel of $L_+$ is $\mathrm{span}\{\partial_j Q\}_{j=1}^3$,
and the kernel of $L_-$ is $\mathrm{span}\{Q \}$.  Furthermore the threshold $\lambda=1$
is not a resonance for either $L_+$ or $L_-$.
\end{thm}

\begin{rem}
As in \cite{S2}, $\alpha$ is called a resonance of  the self-adjoint
operator $L_{+}$ (resp. $L_-$)
if there exists a nontrivial function $f$   such that   $L_{+} f=\alpha f$ 
(resp. $L_-f=\alpha f$) such that $f\notin  L^2(\mathbb{R}^3 )$ and
\begin{equation} \notag
\int_{\mathbb{R}^3}|f(x)|^2(1+|x|)^{-s}\;dx<\infty,\;\;\;\textit{for all}\;\; s>1.
\end{equation} 
%By examining the real and imaginary parts separately, it suffices  to rule out any {real-valued} functions for the edge $\lambda=1$.
\end{rem}

\begin{rem}
As expected, some nontrivial information about the ground state $Q$ enters the spectral analysis. To minimize technical complications, we employ the approximate solution $\tilde Q$ from \cite{CHS}  as a tool for extracting sharp pointwise estimates. The construction of alternative high-precision approximations of $Q$ with controlled and fully rigorous error estimates is certainly possible, but we shall not delve into this matter here.
\end{rem}

With the help of Theorem \ref{thMain}, we obtain the following full characterization of the spectrum of $L_\pm$.

\begin{cor}[Full characterization of the spectrum of $L_{\pm}$] \label{corT1}
We have $$\sigma_{\mathrm{ess}}(L_+)=[1, \infty) = \sigma_{\mathrm{ess}}(L_-).$$
The edge $\lambda=1$ is not an eigenvalue or a resonance for either $L_+$ or $L_-$. For $L_+$ we have
\begin{align}
\sigma_{\mathrm{dis}}(L_+)= \{  -\gamma, \, 0 \},
\end{align}
where $-\gamma<0$ is simple with corresponding eigenfunction being radial, $C^{\infty}$ and having exponential decay.
The kernel of $L_+$ is $\mathrm{span} \{\nabla Q\}$.  For $L_-$ we have $\sigma_{\mathrm{dis}}(L_-) = \{0\}$ with 
the kernel of $L_-$ being $\mathrm{span} \{Q \}$.  The resolvents of $L_\pm$ are given by
\begin{align}
& \rho(L_+) = \mathbb C \setminus \Bigl( \{ 0 \} \cup \{ -\gamma \} \cup [1, \infty)  \Bigr); \\
& \rho(L_-) = \mathbb C \setminus \Bigl( \{ 0 \}\cup [1, \infty)  \Bigr).
\end{align}
\end{cor}
A schematic plot of the spectrum of $L_{\pm}$ can be found in the figure below.
\begin{center}
\begin{tikzpicture}
\draw        (-5,0)  node[above]{$\sigma(L_+)$};
\draw (-5,0) edge[thick,dashed,-latex] node[at end,above]{$\mathbb{R}$} (5,0);
 \draw     (-1.8,0) node[black, scale=1] {\textbullet} node[above, black!80]{$-\gamma$};
 \draw        (-1.8,0)  node[below]{ simple};
\draw      (0,0) node[black, scale=1] {\textbullet} node[above]{$0$};
\draw        (0,0)  node[below]{$\{\nabla Q$\}};
\draw        (1,0)  node[above]{$1$};
\draw        (3,0)  node[below]{essential spectrum};
\draw[-stealth, black,  line width=0.06cm] (1,0) -- (5,0);

\draw        (-5,-2)  node[above]{$\sigma(L_-)$};
\draw (-5,-2) edge[thick,dashed,-latex] node[at end,above]{$\mathbb{R}$} (5,-2);

\draw      (0,-2) node[black, scale=1] {\textbullet} node[above]{$0$};
\draw        (0,-2)  node[below]{$\{ Q$\}};
\draw        (1,-2)  node[above]{$1$};
\draw        (3,-2)  node[below]{essential spectrum};
\draw[-stealth, black,  line width=0.06cm] (1,-2) -- (5,-2);
\end{tikzpicture}
\end{center}

We now explain the main steps of the proof of Theorem \ref{thMain}. Consider first the operator $L_+$ and the equation
$L_+ u = \lambda u$.
The task is to show for $\lambda \in (0,1]$ the above equation admits no solution in $L^2(\mathbb R^3)$. To do this, we argue by contradiction and assume that there is an $L^2$ solution for
some $\lambda \in (0,1]$. By standard elliptic theory, it follows that $u \in H^m(\mathbb R^3)$
for all $m\ge 1$. In particular $u$ admits a rapidly convergent spherical harmonic expansion
\begin{align}
u = \sum_{l=0}^{\infty} \sum_{|m|\le l} R_{ml}(r) Y_{l,m} (\theta,\phi),
\end{align}
where $R_{ml}(r) =
 \int_{\mathbb S^2} \overline{Y_{l,m}(\theta,\phi)} u(x) d\sigma  $
and 
$Y_{l,m}$ are the $L^2(\mathbb S^2)$-normalized spherical harmonics, namely (cf. \cite{Sphe}) for $0\le \theta \le \pi$, $0\le \phi\le 2\pi$,
\begin{align}
Y_{l,m} (\theta, \phi)  & = \left( \frac {(l-m)! (2l+1) } {4\pi (l+m)! } \right)^{\frac 12} e^{im \phi} P_l^m(\cos \theta)
\notag \\
& = \frac{(-1)^{l+m} }{2^l l!}
\left( \frac {(l-m)! (2l+1)} {4\pi (l+m)!} \right)^{\frac 12} e^{im\phi}
(\sin \theta)^m \left( \frac {d} {d(\cos \theta) } \right)^{l+m} (\sin \theta)^{2l}.
\end{align}

\begin{rem} \label{rem_Expand}
Since $u$ is $C^{\infty}$-smooth, by using the Taylor expansion $u(x) = \sum_{|\alpha|\le k_0}
C_{\alpha} x^{\alpha} + \mathcal O(|x|^{k_0+1})$ and the formula for $R_{ml}$, one can infer
that $R_{ml}(r)$ has a regular local expansion 
when $r\to 0^+$.  For example consider the case $l=1$ and the coefficient for the spherical harmonic $x_1/r$. One observes that 
\begin{align} 
 & \int_{\mathbb S^2} \Bigl( \sum_{|\alpha| \le 2N+3} C_{\alpha} x^{\alpha} \Bigr)  \frac {x_1}  r d \sigma \notag \\
 =& c_1 r+ c_3 r^3 +\cdots +c_{2N+3} r^{2N+3} .
\end{align} 
where all even powers of $r$ vanish due to parity.   As a result we can infer
\begin{align}
R_{m1} (r) = c_1 r+c_3 r^3 + \mathcal O(r^5), \qquad \text{as $r\to 0^+$}.
\end{align}
This simple yet important 
observation will be used when we classify the corresponding solutions having regular behaviour
when $r\to 0^+$. 
\end{rem}

By using the aforementioned spherical harmonics expansion, we are led to the following set of equations arranged in ascending
order of degree of the spherical harmonics:
\begin{align}\label{eLge2}
\begin{cases}
\text{$l=0$}: \quad (-\partial_{rr} -\frac 2 r \partial_r +1-\lambda -3Q^2) R_0=0;  \\
\text{$l=1$}: \quad (-\partial_{rr}-\frac 2 r \partial_r +\frac 2 {r^2} +1-\lambda
-3Q^2) R_1=0;  \\
\text{$l\ge 2$}: \quad ( 
-\partial_{rr}-\frac 2 r \partial_r +\frac {l(l+1)} {r^2} +1-\lambda
-3Q^2) R_l=0. 
\end{cases}
\end{align}
Here $R_0$, $R_1$ and $R_l$ are real-valued\footnote{Note that the equations for $R_{ml}$ depend on
the index $l$ only. Hence we suppress the notational dependence on $m$. Also it is enough to work
with real-valued functions by taking separately real and imaginary parts.}
 functions of $r$ only.  
 
From the discussion in Remark \ref{rem_Expand} and the fact that $0$ is a regular singular point of $R_j\,(j\ge 0)$, we have
\begin{align}\label{expand}
\begin{cases}
\mbox{1) $R_j \in C^{\infty}((0,\infty))$ and $R_j \in L^2([0,\infty), {r^2}dr)$;} \\
\mbox{2) $R_j$ has a regular local expansion when $r\to 0^+$.}
\end{cases}
\end{align}

We discuss several cases.

\textsf{The case $l\ge 2$}.  By using
 the pointwise inequality
$\frac {6} {r^2} > 3Q^2(r)$, $\forall\, r>0$ (see Lemma \ref{lem_Qa1}), we rule out
any nontrivial solution to \eqref{eLge2} in $L^2([0,\infty), {r^2}dr)$. 

\textsf{The case $l=0$}. 
Denote $\epsilon=1-\lambda$, $t=r$ and $F_{\epsilon}(t) =t R_0(t)$.
It suffices to consider 
\begin{align}
\begin{cases}
F_{\epsilon}^{\prime\prime} = (\epsilon -3 Q^2) F_{\epsilon}, \quad t>0; \\
F_{\epsilon}(0)=0, \quad F_{\epsilon}^{\prime}(0)=-1.
\end{cases}
\end{align}
By a comparison argument (see Lemma \ref{lem6.1a}), we show that $F_{\epsilon}$ must change sign at least once, and the 
first positive zero $t_{\epsilon}$ of $F_{\epsilon}$ satisfies
\begin{align}
t_{\epsilon} \ge t_0>0,
\end{align}
where $t_0$ is the first positive zero of $F_0$.  We then focus on analyzing the behavior of the solution
after its first positive zero. For this it is enough to study the time-shifted equation
\begin{align}
\begin{cases}
\tilde F_{\epsilon}^{\prime\prime} = (\epsilon- 3 Q^2(t+t_{\epsilon} ) ) \tilde F_{\epsilon},\quad
t>0;\\
\tilde F_{\epsilon}(0)=0, \quad \tilde F_{\epsilon}^{\prime}(0)=1.
\end{cases}
\end{align}
Introduce $q$ solving
\begin{align}
\begin{cases}
q^{\prime\prime} = -3Q^2(t+t_0) q, \quad t>0; \\
q(0)=0, \quad q^{\prime}(0)=1.
\end{cases}
\end{align}
We show via comparison arguments (see  Theorem \ref{thm6.1a}) that $q(t)$ is positive for $t>0$, and $q(t)/t$ stays bounded from below  by a positive constant for $t\in [1,\infty)$. 
Thanks to another comparison argument, we deduce $\tilde F_{\epsilon}(t) \ge q(t)$
for all $t>0$. This yields the desired conclusion for $l=0$.  Quite interestingly, in some sense
we are able to reduce the original $\lambda$-dependent problem to the study of the $\lambda=1$
case.

\textsf{The case $l=1$}.  This is an involved case since 
for $\lambda =0$,  $R=-Q^{\prime}(r)$ solves the equation
\begin{align}
(-\partial_{rr}-\frac 2 r \partial_r +\frac 2 {r^2} +1
-3Q^2) R=0.
\end{align}
If one adopts the Wronskian strategy in this case then one must deal with\footnote{One 
possible fix is to  work with $W(\lambda)/\lambda$.} the degeneracy of
$W(\lambda)$ as $\lambda \to 0$.  In our \textsf{comparison-based} approach, we first use
a local analysis together with suitable normalization to deduce that
\begin{align}
\mathrm{const} \cdot R_1(r) =  r+
{0.1} (1-\lambda-3Q^2(0))r^3 +\mathcal O(r^5), \quad \text{as $r\to 0^+$}.
\end{align}
Denote $t=r$ and $F_{\lambda}(t) =\mathrm{const}\cdot t R_1(t)$. Then 
$F_{\lambda}$ solves
\begin{align}
F_{\lambda}^{\prime\prime} = (1-\lambda + \frac 2 {t^2}
-3 Q^2(t) ) F_{\lambda}, \qquad 0<t<\infty;
\end{align}
and $ F_{\lambda}(t) = t^2 + {0.1}(1-\lambda-3Q^2(0) ) t^4 + \mathcal O(t^6)$, as $t\to 0^+$. 
By a comparison argument (see Proposition \ref{prop3.1a}), we show that $F_{\lambda}$ must change its sign and  the first positive zero $t_0$ of $F_{\lambda}$ satisfies $t_0\ge 0.2$. It then suffices 
to study the solution for $t\ge t_0$. In Proposition \ref{prop4.1a} we show via a further {comparison}
argument that the
corresponding solution must grow in time. 

The above concludes the analysis for the operator $L_+$.  For $L_-$ the analysis is similar and slightly
simpler.
The governing equations are
\begin{align}
&\text{$l=0$}: \quad (-\partial_{rr} -\frac 2 r \partial_r +1-\lambda -Q^2) R_0=0; \label{eNL0}\\
&\text{$l\ge 1$}: \quad ( 
-\partial_{rr}-\frac 2 r \partial_r +\frac {l(l+1)} {r^2} +1-\lambda
-Q^2) R_l=0. \label{eNL1}
\end{align}
By Lemma \ref{lem_Qa1}, we have 
$\frac{2}{r^2} >Q^2(r)$ for all $r>0$. Thus the equation
\eqref{eNL1} does not admit any nontrivial solution in $L^2([0, \infty), {r^2} dr)$. 
For \eqref{eNL0} we show in Theorem \ref{thm7.1a} that it does not admit any nontrivial
$L^2([0,\infty), {r^2}dr)$ solution for $\lambda \in (0,1]$.  The overall strategy is similar to the $L_+$ case.

The rest of this paper is organized as follows. In Section \ref{sec_sturm}, we recall a basic Sturm-type comparison lemma for ODEs and collect several useful properties of the ground state $Q$. Section \ref{sec_pf} is devoted to the proofs of Theorem \ref{thMain} and Corollary \ref{corT1}. The technical estimates used in the proof of Theorem \ref{thMain} for $L_+$ are presented in Sections \ref{sec_4} through \ref{sec_6}. Section \ref{sec_Lminus} delineates the modifications necessary for treating the operator $L_-$. Auxiliary technical estimates are gathered in the appendices.

\subsection*{Notation.} The following notation will be used in this paper.

\begin{itemize}

\item For a nonnegative quantity $Y$, we write $X=\mathcal O (Y)$ if $|X| \le C Y$ for some positive constant $C$.
For any two positive quantities $X$ and $Y$, we write $X\lesssim Y$
or $Y\gtrsim X$
 if $X\le CY$ for some positive constant $C$.
We write $X\lesssim_{Z_1,Z_2,\cdots Z_k} Y$ if $X \le CY$ and the constant $C$ depends on the quantities
$Z_1,\cdots, Z_k$.

\item We define $$
\eta_0(t):=7\cdot 10^{-5}\;\frac{e^{-t}}{1+t}.
$$

\item For simplicity of presentation, decimals (e.g., 0.45) denote exact rational numbers
(e.g., 45/100). All computations in this work are performed using exact rational arithmetic which is fully rigorous. The use of floating-point numbers is strictly avoided to ensure complete rigor.

\end{itemize}

\subsection*{Acknowledgement}
D. Li is supported in part by HKRGC 17310825 and  NSFC 12271236. K. Yang is supported by the Jiangsu Shuang
Chuang Doctoral Plan and the Jiangsu Scientific Research Center of Applied Maths under Grant BK20233002. 
We  thank the anonymous referees for
their helpful comments and suggestions.

\section{Sturm comparison and properties of the ground state $Q$}\label{sec_sturm}
We record the following variation of the standard Sturm-type comparison lemma.
We include a simple
proof for the sake of completeness.
\begin{lem}[Sturm comparison, \cite{Ince, K}] \label{lem2.1a}
Let $0<l_0< \infty$. Suppose $G=G(t)$, $g=g(t)$: $[0,l_0] \to \mathbb R$ are Lipschitz functions satisfying
\begin{align}
G(t) \ge g(t), \qquad \forall\, 0\le t \le l_0.
\end{align}
Assume $F$, $f$ are $C^2([0,l_0])$ functions satisfying 
\begin{align}
\begin{cases}
F^{\prime\prime} = G F,  \quad 0<t<l_0;\\
f^{\prime\prime} = g f, \quad 0<t<l_0;\\
F(0)=f(0)\ge 0, \quad F^{\prime}(0)\ge f^{\prime}(0),
\end{cases}
\end{align}
and $f(t) >0$ for all $0<t <l_0$. Then 
\begin{align}
F(t) \ge f(t)>0,\;\;\;{ \frac {F^{\prime}(t)} {F(t) }\geq \frac {f^{\prime}(t)} {f(t) },} \qquad \forall\, 0<t < l_0.
\end{align}
If in addition $F(0)=f(0)>0$ and $F^{\prime}(0)> f^{\prime}(0)$, then we have the strict inequality, namely
\begin{align}
F(t)>f(t)>0, \quad \frac{F^{\prime}(t)} {F(t)} >
\frac{f^{\prime}(t)}{f(t)},  \qquad \forall\, 0<t<l_0.
\end{align}
If $F(0)=f(0)=0$ and $F^{\prime}(0) > f^{\prime}(0)$, then
\begin{align}
F(t)>f(t)>0, \qquad \forall\, 0<t<l_0.
\end{align}
\end{lem}
\begin{rem}
The same conclusion holds if $f(t)>0$ for all $0\le t<l_0$ and
\begin{align}
F(0)\ge f(0)>0, \qquad
\Bigl( \frac {F^{\prime}} {F }  -\frac {f^{\prime} }f \Bigr)\Bigr|_{t=0} \ge 0.
\end{align}
\end{rem}

\begin{proof}
We sketch the (standard) argument. 
Consider first the case $F(0)=f(0)>0$. Note that for this case it is enough to prove the theorem
under the assumption that $F(t)>0$ for all $0<t<l_0$. Once this is proved, the general
case follows by a simple bootstrapping argument. 

Denote $R=R(t) = \frac {F^{\prime}(t)} {F(t)}$, $r=r(t) =\frac {f^{\prime}(t) } {f(t) }$.
Clearly $(R-r)\Bigr|_{t=0} \ge 0$.  Then
\begin{align}
(R-r)^{\prime} &=
\frac {F^{\prime\prime} F - (F^{\prime})^2} {F^2}
- \frac {f^{\prime\prime} f - (f^{\prime})^2} {f^2} \\
& = G- g  - R^2+r^2 \\
& \ge - (R+r) (R-r).
\end{align}
Integrating in time then yields that $R-r\ge 0$ for all $t$.  Thus
\begin{align}
R-r = \Bigl( \log \frac {F(t)} {f(t)} \Bigr)^{\prime} \ge 0.
\end{align}
Thus $F(t) \ge f(t)$ for all $0\le t< l_0$.

Next consider the case $f(0)=0$.  Clearly $F^{\prime}(0)\ge f^{\prime}(0)>0$ (if $f^{\prime}(0)=0$
then $f\equiv 0$ which is a contradiction to the assumption $f>0$ for $t>0$).  By continuity we have
$F(t)>0$ for $t \in (0,t_1)$, where $t_1>0$ is sufficiently small.

Define for $\epsilon>0$,  $f_{\epsilon}=f+\epsilon$.  Let $F_{\epsilon}$ solve
\begin{align}
\begin{cases}
F^{\prime\prime}_{\epsilon} = G \frac f {f_{\epsilon}} F_{\epsilon}, \\
F_{\epsilon}(0)=\epsilon, \qquad F_{\epsilon}^{\prime}(0)=F^{\prime}(0)
\end{cases};
\qquad
\begin{cases}
f_{\epsilon}^{\prime\prime} = g \frac f {f_{\epsilon}} f_{\epsilon},\\
f_{\epsilon}(0) =\epsilon, \qquad f_{\epsilon}^{\prime}(0)=f^{\prime}(0).
\end{cases}
\end{align}
Apparently, 
\begin{align}
F_{\epsilon}(t) \ge f_{\epsilon}(t) >0, \qquad
\frac{ F_{\epsilon}^{\prime} (t) } {F_{\epsilon}(t) }  \ge 
\frac {f_{\epsilon}^{\prime}(t)} {f_{\epsilon}(t)}, \qquad\forall\, 0<t< l_0.
\end{align}
Sending $\epsilon\to 0^{+}$ then yields the desired inequality (note 
that\footnote{Observe that $\epsilon/f_{\epsilon}(t) \lesssim \frac{\epsilon} {\epsilon +t}$ for $t\lesssim 1$
since $f(0)=0$, $f^{\prime}(0)>0$.}
$F_{\epsilon}(t) \to F(t)$ and $F_{\epsilon}^{\prime}(t) \to F^{\prime}(t)$ as $\epsilon \to 0^+$ for each $t$).

Finally if $F^{\prime}(0)>f^{\prime}(0)$ and $F(0)=f(0)>0$,
 the strict inequality follows easily from
the fact that $R(t)-r(t) >0$ for all $0\le t<l_0$.

If $F^{\prime}(0)>f^{\prime}(0)$ and $F(0)=f(0)=0$, then  using $R-r=(\log \frac {F(t) } {f(t)} )^{\prime}\ge 0$,
we obtain
\begin{align}
\frac {F(t)} {f(t)} 
\ge \lim_{s \to 0^+} \frac {F(s) }{f(s)} = \frac {F^{\prime}(0)} {f^{\prime} (0) }>1.
\end{align}
Thus $F(t) >f(t)>0$ for all $0<t<l_0$.
\end{proof}

The following lemma collects a few important properties of $Q$ which will be needed later.
In the proof we shall use the Costin-Huang-Schlag approximate solution $\tilde Q$ whose properties
are reviewed in Appendix \ref{AppQ1}.

\begin{lem}[Properties of $Q$]   \label{lem_Qa1}
The following hold.
\begin{enumerate}
\item The Lyapunov functional (cf. \cite{C}) $E(t)=\frac{Q'(t)^2}{2}-\frac{Q(t)^2}{2}+\frac{Q(t)^4}{4}$ is decreasing in $t$.  We have the derivative estimate
$|Q^{\prime}(t)| \le \sqrt{ (Q^{\prime}(s))^2 + \frac 12 Q(s)^4}$ for any $0\le s \le t<\infty$.
In particular $|Q^{\prime}(t)| \le \frac 1 {\sqrt 2} Q(0)^2$. 
\item For all $t>0$, $\frac{Q'(t)}{t}$ is increasing in $t$. Also  $|Q''(t)|< 81.7$ for all $t\ge 0$. 
\item $\frac 2 {t^2} - 3 Q^2(t)>0$ if $0<t\le 0.2$ or $t\ge  1.2$. Also
$\frac 2 {t^2} - 3 Q^2(t)> -13.64$, for any $t\in [0.2, 1.2]$, and
$\frac{2}{t^2}-3 Q^2(t)\ge -6.04$ for any  $t\in [0.625,1.2]$.
\item $\frac {6} {t^2}  -3Q^2 >0$ for any $t>0$.

\end{enumerate}
\end{lem}
\begin{proof}[Proof of Lemma \ref{lem_Qa1}]
(1) This follows from the identity
\begin{align} \notag
E^{\prime}(t)=-\frac 2 t (Q^{\prime}(t) )^2.
\end{align}

(2) Clearly
\begin{align} \notag
\frac d {dt} (\frac 1 t Q^{\prime}) = -\frac 1 {t^2} Q^{\prime}+\frac 1 t Q^{\prime\prime}
= - 3 t^{-2} Q^{\prime} + t^{-1} (Q-Q^3).
\end{align}
Let $t_1>0$ be the unique point such that $Q(t_1)=1$, where the uniqueness of $t_1$ follows from the strict monotonicity of $Q$ in \cite{C}. Clearly $Q-Q^3>0$ for any $t>t_1$.  Thus
we only need to show $(t^{-1} Q^{\prime})^{\prime} >0$ for $0\le t \le t_1$. To this end, note that
\begin{align} 
  (t^2 Q^{\prime} )^{\prime} &= t^2 (Q-Q^3)  \notag \\
  \Rightarrow  t^2 Q^{\prime}(t)  &= \int_0^t s^2 (Q(s) -Q(s)^3) ds \notag \\
  & < \int_0^t s^2 ds (Q(t) -Q(t)^3)  \qquad (\text{we used $(Q-Q^3)^{\prime}=
  (1-3Q^2) Q^{\prime}>0$ for $s \in [0, t_1]$}) \notag \\
  & =\frac 13 t^3 (Q(t)-Q(t)^3).
  \end{align}
  Thus $(t^{-1} Q^{\prime})^{\prime} >0$ for $t\in [0,t_1]$ as well.

Since $\frac{Q' }{t}$ is increasing when $t>0$, we have 
$$
-Q(0)^3\le Q''(t)=Q(t)-Q(t)^3-\frac{2Q'(t)}{t}\le Q(0)-2Q''(0)=Q(0)-\frac{2}{3} (Q(0)-Q(0)^3).
$$
Thus, $|Q''(t)|\le (\tilde Q(0)+\eta_0(0))^3<81.7.$

(3) For $0<t\le 0.2$, thanks to the explicit expression of $\tilde Q(t)$, one can check  that
\begin{align}
\frac 2 {t^2} - 3 (\tilde Q(t))^2 >2.
\end{align}
Indeed, on $(0,0.18]$, 
\begin{align*}
\frac 2 {t^2} - 3 (\tilde Q(t))^2-2\geq \frac 2 {(\tfrac{18}{100})^2} - 3 (\tilde Q(0))^2-2>1;
\end{align*} 
On $[0.18, 0.2]$, 
\begin{align*}
\frac 2 {t^2} - 3 (\tilde Q(t))^2-2\geq \frac 2 {(\tfrac{1}{5})^2} - 3 (\tilde Q(\tfrac{18}{100}))^2-2>1.
\end{align*}

Denote $\mathcal{E}=  Q - \tilde Q$ and recall that $\|\mathcal{E}\|_{\infty}<7\times 10^{-5}$ by \eqref{D6}, we have
\begin{align} \label{bd0.1o}
|3Q^2 -3 (\tilde Q)^2 | \le 3\|\mathcal{E} \|_{\infty}^2 +6|\tilde Q(0) |  \cdot \|\mathcal{E} \|_{\infty}<\tfrac{11211}{6154000}<1.
\end{align}  
Thus $2 t^{-2} -3 Q^2>0$ for $0<t\le 0.2$. 

By Proposition \ref{prop_Qs1}, we have 
\begin{align}
\frac{187}{69} \frac {e^{-t}} t < \tilde Q (t) < \frac{350}{129} \frac {e^{-t}}t, 
\quad \forall\, t\ge 2.5; \quad |Q(t) -\tilde Q(t) | \le 7\cdot 10^{-5}\cdot \frac {e^{-t}}{1+t}=:\eta_0(t), \quad\forall\, t\ge 0.
\end{align}
It is not difficult to verify for $t\ge 2.5$, 
\begin{align} \label{4.7aa}
\frac 2 {t^2} - 3 \cdot ( \frac{350}{129} \cdot \frac 1 t e^{-t} + \eta_0(t))^2 >\frac 2 {t^2} - 3 \cdot ( \frac{351}{129} \cdot \frac 1 t e^{-t})^2> \frac{1} {t^2}.
\end{align}
Thus $2 t^{-2} -3Q^2>0$  for $t\ge 2.5$.

For the regime $1.2\le t \le 2.5$ we resort to rigorous computation. We partition the  interval $[1.2,2.5]$ into $N=130$ intervals $[t_i,t_{i+1}]$ with equal length $t_{i+1}-t_i =\frac{1}{100}$ ($0\le i \le  N-1$). From \eqref{bd0.1o}, it is straightforward to check (below we recall $\eta_0 =7\cdot 10^{-5} e^{-t} /(1+t)$)
$$
\min_{0\leq i\le N}\{\frac 2 {t_i^2}-3 Q^2(t_i)\}\ge \min_{0\le i\le N} \{\frac 2 {t_i^2}-3 (\tilde Q )^2(t_i)- \tfrac{11211}{6154000}\}\ge \tfrac{23983}{37847} \cdot 10^{-1}.
$$

Clearly, $$\max_{t\in[1.2,2.5]} \frac 4 {t^3}=\frac 4 {{1.2}^3} =\frac{125}{54}.$$ 
Since the Lyapunov functional $E(t)$ is decreasing, we have 
\begin{align}
\max_{t \in [1.2, 2.5]} |Q^{\prime}(t) | &\le \sqrt{ (Q^{\prime}(\tfrac{6}{5}) )^2 + \frac 12 Q(\tfrac{6}{5})^4 } \le \sqrt{ (\tilde Q^{\prime}(\tfrac{6}{5}) - 5\cdot \eta_0(\tfrac{6}{5}))^2+\frac 12 (\tilde Q(\tfrac{6}{5}) +\eta_0(\tfrac{6}{5}) )^4}\notag \\
 &  \le \sqrt{ (\tilde Q^{\prime}(\tfrac{6}{5}) - 5\cdot \eta_0(0))^2+\frac 12 (\tilde Q(\tfrac{6}{5}) +\eta_0(0) )^4}\leq \tfrac{1203}{1000}. 
\end{align}
Hence 
\begin{align}
&\max_{t\in[1.2,2.5]} (-6 Q  Q^{\prime} )\le 6 (\tilde Q(\tfrac{6}{5})+\eta_0(0)) \cdot  \tfrac{1203}{1000}\le \tfrac{88512}{18484}.
\end{align}

Note that (below we use $|a-b| \le \max\{a,b\}$ for $a>0$, $b>0$)
\begin{align*}
\max_{t\in[1.2,2.5]} |(\frac{2}{t^2}-3 Q^2(t))'|&=\max_{t\in[1.2,2.5]} | -\frac 4 {t^{3}}-6Q(t)Q'(t)|\le \max\{\max_{t\in[1.2,2.5]} \frac 4 {t^3},\max_{t\in[1.2,2.5]} -6 Q(t)Q'(t)\} \\
& \leq \max (\tfrac{125}{54},\tfrac{88512}{18484} )=\tfrac{88512}{18484}.
\end{align*}

Thanks to the above derivative estimate, we have for any $t\in[t_i,t_{i+1}]$ and $0\le i\le N-1$,
$$
\frac{2}{t^2}-3 Q^2(t)-(\frac{2}{t_i^2}-3 Q^2(t_i))
\geq -\tfrac{88512}{18484} \Delta t=-\tfrac{88512}{18484}\cdot 10^{-2}. $$
Thus, for $1.2\le t \le 2.5$, we have
$$
\frac{2}{t^2}-3 Q^2(t)\ge \min_{0\le i\le N}\{\frac 2 {t_i^2}-3 Q^2(t_i)\}-\tfrac{88512}{18484}\cdot 10^{-2}\ge \tfrac{23983}{37847}\cdot 10^{-1}-\tfrac{88512}{18484}\cdot 10^{-2}>0.
$$
Thus we obtain $2 t^{-2} -3Q^2(t)>0$ for $0<t\le 0.2$ or $t\ge 1.2$. The statement
$2 t^{-2}-3Q^2(t)>-13.64$ for any $t\in [0.2, 1.2]$, $2t^{-2} -3Q^2(t)>-6.04$ for any
$t\in [0.625, 1.2]$ can be checked along similar lines. We omit the details.

(4) The proof of this inequality is similar to those in (3). We omit the details.
\end{proof}

\section{Proof of Theorem \ref{thMain} and Corollary \ref{corT1}}\label{sec_pf}
In this section, we prove Theorem \ref{thMain} and Corollary \ref{corT1}. To streamline the presentation, the proof invokes several technical estimates that are established in later sections.

\begin{proof}[Proof of Theorem \ref{thMain}]
We discuss several cases.

 1) \textsf{Eigenvalue case of $L_+$.}
We consider the equation $L_+ f = \lambda f$ for $\lambda \in [0,1]$ and $f\in L^2(\mathbb R^3)$. By standard
elliptic estimates one has $f \in H^m(\mathbb R^3)$ for all integer $m\ge 1$. We decompose $f$  into spherical harmonics as in \eqref{eLge2}, satisfying \eqref{expand}.

\textsf{Subcase $l\ge 2$}.  
By Lemma \ref{lem_Qa1}, we have the pointwise bound (note that $l\ge 2$)
\begin{align} 
\frac {l(l+1) } {r^2} \ge \frac {6}{r^2} > 3Q^2(r), \qquad \forall\, r>0.
\end{align}
It follows that the above system cannot admit any nontrivial $L^2$ solution. In the following, we present an 
additional argument to rule out the resonance case.
By Remark \ref{rem_Expand}, it is not difficult to check that for $l \ge 2$:
\begin{align}
R_l(r)= \mathcal O(r^2), \quad R_l^{\prime}(r) = \mathcal O(r), \quad \text{as $r\to 0^+$}.
\end{align}
Denote $F(r)=r R_l(r)$ and $V(r)=\frac{l(l+1)} {r^2} +1-\lambda-3Q^2$. Note that 
$0<V(r) \lesssim 1+ r^{-2}$. Clearly
\begin{align}
 & F^{\prime\prime} =V F, \qquad r>0; \notag \\
 \Rightarrow & \frac 12 (F^2)^{\prime\prime} = (F^{\prime})^2 + V F^2, \quad r>0; \notag\\
 \Rightarrow & \frac 12 \int_0^{\infty} F^2 \psi_A^{\prime\prime} dr=
 \int_0^{\infty} ( (F^{\prime})^2 + V F^2) \psi_A dr,   \label{3.6Aa1}
 \end{align}
where $\psi_A(r) =\psi( \frac r A)$, and $\psi \in C_c^{\infty}(\mathbb R)$ is such that
$\psi(z)\equiv 1$ for $|z|\le 1$, $\psi(z)=0$ for $|z|\ge 2$, and $0\le \psi (z) \le 1$ for all $z \in \mathbb R$.
Note that the regularity $F(r)=\mathcal O(r^3)$, $F^{\prime}(r)=\mathcal O(r^2)$ as $r\to 0^+$ is used
when we derive $\int_0^{\infty} F^{\prime\prime} \psi_A dr = \int_0^{\infty} F \psi_A^{\prime\prime} dr$.
Taking $A\to \infty$ easily yields that $F\equiv 0$.  Note here in order to have  $\int_0^{\infty}
F^2 \psi_A^{\prime\prime} dr \to 0$ as $A\to \infty$, it suffices to require $ A^{-2} \int_{0.01 A
\le r \le 100 A} F^2(r) dr \to 0$. This is obviously fulfilled if
\begin{align} \label{3.7Aa1}
\int_{0}^{\infty} \frac {F^2(r)} {(1+r)^s} dr <\infty, \qquad \text{for all $s>1$}.
\end{align}

\textsf{Subcase $l=1$ and $\lambda=0$}.
In this case by Remark \ref{rem_Expand}, we have $R_1(r) = c_1 r+c_3r^3+\mathcal O(r^5)$ as
$r\to 0^+$.  Denote $F(r) =r R_1(r)$.  Note that
\begin{align}
F^{\prime\prime} = (\frac 2 {r^2} +1-3Q^2) F, \qquad r>0. \notag
\end{align}
One solution is given by $F_0(r) = -r Q^{\prime}(r)$. We shall prove uniqueness. Namely if $F(r)$ is another solution
with the asymptotics $F(r) = c_1 r^2 +c_3 r^4 +\mathcal O(r^6)$ as $r\to 0^+$, then let
$\theta(r)= F(r) + c_1 \frac 1 {Q^{\prime\prime}(0) } F_0(r)$. Obviously
\begin{align}
& \theta^{\prime\prime}= (\frac 2 {r^2} +1-3Q^2) \theta, \qquad r>0; \notag \\
&\theta(r) = \mathcal O(r^4), \qquad \text{as $r\to 0^+$}. \label{eq_The1}
\end{align}
By Proposition \ref{propA1a} in the appendix, we conclude $\theta \equiv 0$. Therefore 
$R_1(r)=-Q^{\prime}(r)$ is the only solution (up to a multiplicative constant). Note that
$R_1(r)=-Q^{\prime}(r)$ corresponds to the statement that the kernel of $L_+$
is $\mathrm{span}\{\nabla Q \}$.

\textsf{Subcase $l=1$ and $\lambda \in (0,1]$}. 
In this case denote $F_{\lambda} (r) =r R_1(r)$ and note that
\begin{align}
F_{\lambda}^{\prime\prime} = (\frac 2 {r^2} +1-\lambda -3Q^2) F_{\lambda}, \qquad r>0. \notag
\end{align}
By Remark \ref{rem_Expand}, we have $F_{\lambda}(r)= c_1 r^2+ c_3 r^4+\mathcal O(r^6)$ as
$r\to 0^+$.  By a similar reasoning as in \eqref{eq_The1}, we infer $c_1 \ne 0$ if $F_{\lambda}$ is not
identically zero. With no loss
we can take $c_1=1$ and work out $c_3=0.1(1-\lambda-3Q^2(0) )$.  At this point we can invoke 
Proposition \ref{prop3.1a} to conclude that $F_{\lambda}$ must change its sign and the first
positive zero $t_0$ satisfies $t_0>0.625$. By Proposition \ref{prop4.1a}, we conclude that 
$F_{\lambda}(r)$ must grow without bound as $r\to \infty$. Thus we arrive at a contradiction (since $F_{\lambda}$
is assumed to be in $L^2([0, \infty), dr)$) and $F_{\lambda}$
must be identically zero.

\textsf{Subcase $l=0$.}  In this case we denote $\epsilon=1-\lambda \in [0,1]$,
$F_{\epsilon}(r)= r R_0(r)$  and consider the equation\footnote{If $F_{\epsilon}^{\prime}(0)=0$, then
clearly $F_{\epsilon} \equiv 0$. We only need to rule out the case $F_{\epsilon}^{\prime}(0) \ne 0$.
Hence for simplicity we take $F_{\epsilon}^{\prime}(0)=-1$.}
\begin{align}
\begin{cases}
F_{\epsilon}^{\prime\prime} = (\epsilon -3 Q^2 ) F_{\epsilon}, \quad r>0;\\
F_{\epsilon}(0)=0, \; F_{\epsilon}^{\prime}(0)=-1.
\end{cases}
\end{align}
By Lemma \ref{lem6.1a}, the function $F_{\epsilon}$ must change its sign with its first positive zero
$t_{\epsilon}\ge t_0$ for all $\epsilon \in [0,1]$. Here $t_0$ is the first positive zero of $F_0$. 
By Theorem \ref{thm6.1a}, we conclude that $F_{\epsilon}$ must be growing as $r\to \infty$. 
Thus this case is settled.

2) \textsf{We show $\lambda=1$ is not a resonance for $L_+$}. 
We assume $L_+f =f$ where $f \not \in L^2(\mathbb R^3)$ and
\begin{equation}\notag
\int_{\mathbb{R}^3}|f(x)|^2(1+|x|)^{-s}\;dx<\infty,\;\;\;\textit{for all}\;\; s>1.
\end{equation} 
By examining the equation $-\Delta f =3Q^2 f$,  it is not difficult to check that  $\nabla f \in H^m(\mathbb R^3)$
for any integer $m\ge 1$.  In particular $f \in C^{\infty}(\mathbb R^3)$ and the asymptotics near $r\to 0^+$ still
holds in the sense of Remark \ref{rem_Expand}.  We then proceed to examine the coefficient of $f$ near each
spherical harmonics, namely the set of equations
\begin{align}
\text{$l\ge 0$}: \quad ( 
-\partial_{rr}-\frac 2 r \partial_r +\frac {l(l+1)} {r^2} -3Q^2) R_l=0,
\end{align}
where  $R_l\;(l\ge 0)$ are real-valued functions of $r$ only.   Each $R_l \in C^{\infty}( (0, \infty) )$ and
has regular expansion as $r\to 0^+$ in the sense of Remark \ref{rem_Expand}. Moreover,
$R_l$ satisfies the integrability condition:
\begin{equation}\label{res}
\int_{0}^{\infty}|R_l(r)|^2 r^2(1+r)^{-s}\;dr<\infty\;\;\;\; \textit{for all\;\;}s>1.
\end{equation}

\textsf{Subcase $l\ge 2$}.  This is already proved by the discussion around \eqref{3.7Aa1}. In this case
we must have $R_l\equiv 0$.

\textsf{Subcase $l=1$}.  This is again covered by Proposition \ref{prop4.1a}.  We have
$R_l \equiv 0$.

\textsf{Subcase $l=0$}. This is covered by Theorem \ref{thm6.1a}. We also have $R_0\equiv 0$ in this
case.

Thus $\lambda=1$ is not a resonance for $L_+$.

3) \textsf{Eigenvalue case of $L_-$.}  The discussion for $L_-$ is similar to that of $L_+$. We sketch
the details in Section \ref{sec_Lminus}.  In particular we use  Theorem \ref{thm7.1a} for $l=0$ and 
(\ref{L-2}) for $l\ge 1$. The fact that the kernel of $L_-$ is $\mathrm{span}\{Q \}$ is proved in Remark
\ref{rem7.1a} for $l=0$ together with the non-existence result for $l\ge 1$.

4) \textsf{$L_-$ has no resonance at the edge $\lambda=1$.} The discussion is similar to the $L_+$ case
and in fact simpler. For $l=0$ this is effectively proved by Theorem \ref{thm7.1a}. For $l\ge 1$ this is due
to \eqref{L-2}. 
\end{proof}

\begin{proof}[Proof of Corollary \ref{corT1}]
We focus on showing that the existence of a simple eigenvalue $-\gamma$
for $L_+$.  The situation for $L_-$ is simpler and omitted.

 It suffices to carry out the analysis of the equation $L_+ f= \lambda f$ for $\lambda<0$ and $f\in L^2(\mathbb R^3)$.  By standard
elliptic estimates one has $f \in H^m(\mathbb R^3)$ for all integer $m\ge 1$. We decompose $f$  into spherical harmonics as in \eqref{eLge2}, satisfying \eqref{expand}.

\textsf{Subcase $l\ge 2$.}  By the same analysis as in \eqref{3.6Aa1}, we obtain $R_l \equiv 0$ for $l\ge 2$.

\textsf{Subcase $l=1$.}  By Proposition \ref{prop3.2a}, if $F_{\lambda}(r)= r R_1(r)$ is nontrivial, then it 
must change its sign at some $t_0>0.625$.  By Proposition \ref{prop4.1a},  $F_{\lambda}$ must be growing
as $r\to \infty$ which is an obvious contradiction. Thus $R_1(r) \equiv 0$ in this subcase.

\textsf{Subcase $l=0$.}  In this case we define $F_{\lambda}(r)= r R_0(r)$ and consider the equation
\begin{align}
\begin{cases}
F_{\lambda}^{\prime\prime} = (1-\lambda-3Q^2) F_{\lambda}, \notag \\
F_{\lambda}(0)=0, \; F_{\lambda}^{\prime}(0)=1.
\end{cases}
\end{align}
By Proposition \ref{propA2b} in the appendix, we obtain that $-\gamma$ is the desired eigenvalue.
\end{proof}

\section{when $0<\lambda\le 1$, the  solution must change sign}\label{sec_4}
In this section we give the technical estimates for the $l=1$ case of $L_+$. 
\begin{lem} \label{lem3.1a}
Suppose $F$ is a smooth function solving the linear equation
\begin{align}
F^{\prime\prime} = (\frac 2 {t^2} -3 Q^2(t) ) F, \qquad 1\le t<\infty.
\end{align}
Then for some constants $c_1$, $c_2$ we have
\begin{align}
F(t) = c_1 (t^2 +\eta_1(t)) + c_2 ( \frac 1 t + \eta_2(t) ),
\end{align}
where $\eta_i(t)$ are smooth functions satisfying
\begin{align}
\sup_{1\le t <\infty} ( | e^t \eta_1(t) | + | e^t \eta_2(t)| )<\infty.
\end{align}
\end{lem}
\begin{proof}
It suffices   to exhibit two independent solutions. We consider $\eta_1$ solving the
integral equation
\begin{align}
\eta_1(t) = \int_t^{\infty} (s-t) \Bigl( -3 Q^2(s) s^2 + (\frac 2 {s^2}-3Q^2(s) ) \eta_1(s) 
\Bigr)ds, \qquad t\ge T_1.
\end{align}
By taking $T_1$ sufficiently large, one can obtain a contraction in the norm 
$\|e^t \eta_1(t) \|_{L_t^{\infty}([T_1,\infty) )}$.  Clearly the function
$\Theta_1(t) = t^2 + \eta_1(t)$ solves the original ODE  on $(T_1,\infty)$.  Solving it backward
in time and noting that it is a linear equation, we obtain a smooth solution $\Theta_1(t)$
defined on $[1,\infty)$. 

Analogously we can find $\eta_2$ solving
\begin{align}
\eta_2(t) = \int_t^{\infty} (s-t) \Bigl( -3 Q^2(s) \frac 1s+ (\frac 2 {s^2}-3Q^2(s) ) \eta_2(s) 
\Bigr)ds, \qquad t\ge T_2.
\end{align}
The second solution $\Theta_2(t) = \frac 1 t + \eta_2(t)$ on $[1,\infty)$ is also easily obtained.

To check the independence of the two solutions one can examine the Wronskian.  It is clearly
nonzero for large $t$ and hence nonzero for all $t$.
\end{proof}

\begin{prop} \label{prop3.1a}
Suppose $0<\lambda\le 1$ and $F_{\lambda}=F_{\lambda}(t)$ solves
\begin{align}
F_{\lambda}^{\prime\prime} = (1-\lambda + \frac 2 {t^2}
-3 Q^2(t) ) F_{\lambda}, \qquad 0<t<\infty.
\end{align}
To fix the normalization we fix $F_{\lambda}(t)$ such that
\begin{align}
F_{\lambda}(t) = t^2 + 0.1(1-\lambda-3Q^2(0) ) t^4 + \mathcal O(t^6), \quad \text{as $t\to 0^+$}.
\end{align}
Then $F_{\lambda}$ must change its sign at least once on $(0,\infty)$. Moreover the first
positive zero $t_0$ of $F_{\lambda}$ satisfies $t_0>  0.625$.

\end{prop}
\begin{proof}
We first show that $F_{\lambda}$ must change sign on $(0,\infty)$. Assume that $F_{\lambda}$
stays positive (note that $F_{\lambda}$ cannot touch the $x$-axis on $(0,\infty)$ by uniqueness). 
Clearly for $t=0^+$, we have
\begin{align}
&\log F_{\lambda} = 2 \log t + 0.1 (1-\lambda-3Q^2(0) ) t^2 +\mathcal O(t^4); \\
& \frac {F_{\lambda}^{\prime}(t) } { F_{\lambda}(t) }
= \frac 2 t + 0.2 (1-\lambda-3Q^2(0)) t +\mathcal O(t^3). 
\end{align}
In particular it is not difficult to check that for $t_1>0$ sufficiently small, we have (below we used $\lambda>0$)
\begin{align}
&\frac {F_{\lambda}^{\prime}(t) } {F_{\lambda}(t) } < \frac {  \beta^{\prime}(t) } {\beta(t) }, \quad t=t_1; \\
& F_{\lambda}(t_1) < \beta(t_1),
\end{align}
where $\beta(t) = - c_1 t Q^{\prime}(t)$, and $c_1>0$ is sufficiently large. Note that
\begin{align*}
\beta^{\prime\prime} = (1 + \frac 2 {t^2}
-3 Q^2(t) ) \beta.
\end{align*}
Comparing $\beta$ with $F_{\lambda}$ on $[t_1,\infty)$ and 
using the assumption that $F_{\lambda}$  is positive, we obtain
\begin{align}
0<F_{\lambda}(t) \le  \beta(t), \qquad \forall\, t_1\le t<\infty.
\end{align}

First we discuss the case $\lambda=1$. By Lemma \ref{lem3.1a}, 
the solution must decay as $t^{-1}$ as $t\to \infty$. But then it clearly contradicts to
the upper bound $\beta(t)$ which decays as $\mathcal O( e^{-t})$. 

The case $0<\lambda<1$ is similar. One can also obtain a contradiction. Thus
$F_{\lambda}$ must change its sign on $(0,\infty)$.

Clearly, the first positive zero $t_0\ge 0.2$ by Lemma \ref{lem_Qa1} (since $2t^{-2} -3Q^2>0$ for $0<t\le 0.2$). In particular, $F_\lambda(0.2)>0$ and $F_\lambda'(0.2)>0$ for all $\lambda\in (0,1]$. It remains  
to show $t_0>0.625$. We resort to the following comparison argument.

We consider $U_A$ and $U_B$ solving
\begin{align}
\begin{cases}
U_A^{\prime\prime} = (1-\lambda+
\frac 2 {t^2}  - 3 Q^2(t) ) U_A, \quad t>0.2; \\
(U_A,U_A')|_{t=0.2}=(1,0).
\end{cases}
\begin{cases}
U_B^{\prime\prime} = (1-\lambda+
\frac 2 {t^2}  - 3 Q^2(t) ) U_B, \quad t>0.2; \\
(U_B,U_B')|_{t=0.2}=(0,1).
\end{cases}
\end{align}
By Lemma \ref{lem_Qa1}, we have
$$
1-\lambda+\frac 2 {t^2}-3 Q^2(t)\geq \frac 2 {t^2}-3 Q^2(t)>-13.64,\quad \forall \, t\in [0.2,1.2].
$$ 
Denote $k_0=\sqrt{13.64}$. Let $\tilde{U}_A=\cos(k_0(t-0.2))$ and $\tilde U_B=\frac{\sin(k_0(t-0.2))}{k_0}$ solve
\begin{align}
\begin{cases}
\tilde U_A^{\prime\prime} = -13.64 \tilde U_A, \quad t>0.2; \\
(\tilde U_A,\tilde U_A')|_{t=0.2}=(1,0).
\end{cases}
\begin{cases}
\tilde U_B^{\prime\prime} = -13.64 \tilde U_B, \quad t>0.2; \\
(\tilde U_B,\tilde U_B')|_{t=0.2}=(0,1).
\end{cases}
\end{align}
Clearly by Lemma \ref{lem2.1a}, $U_A\geq \tilde U_A>0$ and  $U_B\geq \tilde U_B>0$ on the time interval $[0.2, 0.625]$. As $F_\lambda$ is a positive linear combination of $U_A$ and $U_B$, $F_\lambda$ and $F_\lambda'$ stay positive on $[0,0.625]$. Thus the first positive zero must occur after $0.625$. 
\end{proof}

The following proposition deals with the interesting regime $\lambda<0$.  To guarantee sign-changing, we impose an extra mild growth condition on
 $F_{\lambda}$ as $t\to \infty$. 

\begin{prop} \label{prop3.2a}
Suppose $\lambda<0$ and $F_{\lambda}=F_{\lambda}(t)$ solves
\begin{align}
F_{\lambda}^{\prime\prime} = (1-\lambda + \frac 2 {t^2}
-3 Q^2(t) ) F_{\lambda}, \qquad 0<t<\infty.
\end{align}
To fix the normalization we fix $F_{\lambda}(t)$ such that
\begin{align}
F_{\lambda}(t) = t^2 + 0.1(1-\lambda-3Q^2(0) ) t^4 + \mathcal O(t^6), \quad \text{as $t\to 0^+$}.
\end{align}
\textsf{Assume $t^{-3} F_{\lambda}(t)$ is bounded as $t\to \infty$}. 
Then $F_{\lambda}$ must change its sign at least once on $(0,\infty)$. Moreover the first
positive zero $t_0$ of $F_{\lambda}$ satisfies $t_0>  0.625$.
\end{prop}
\begin{proof}[Proof of Proposition \ref{prop3.2a}]
We assume $F_{\lambda}$ does not change its sign and remain positive on $(0, \infty)$. Clearly
due to the mild growth assumption,  $F_{\lambda}$ cannot be exponentially growing as $t\to \infty$. 
Thus $F_{\lambda}(t) = \mathcal O(e^{-\sqrt{2} t})$ as $t\to \infty$. 
In particular $F_{\lambda}(t)$ decays faster than $\beta(t)= -c_2 t Q^{\prime}(t)$. We
 can guarantee at some $t_1>0$ sufficiently small,
\begin{align}
&\frac {\beta^{\prime}(t_1) } {\beta(t_1) } < \frac {F_{\lambda}^{\prime}(t_1)} { F_{\lambda}(t_1)}, \notag \\
&\beta(t_1)< F_{\lambda}(t_1), \qquad (\text{here we need to take $c_2>0$ sufficiently small}). \notag
\end{align}
By using comparison, we then deduce that $\beta(t) \le F_{\lambda}(t)$ for all $t>t_1$.  This clearly
contradicts to the fact that $F_{\lambda}(t)$ decays faster than $\beta(t)$ as $t\to \infty$. 

Thus $F_{\lambda}$ must change its sign at some $t_0>0$. The estimate $t_0>0.625$ is the same as in the
proof of Proposition \ref{prop3.1a}.
\end{proof}

\section{After the first positive zero}\label{sec_5}

\begin{prop} \label{prop4.1a}
Consider
\begin{align}\label{410}
\begin{cases}
G^{\prime\prime} = (
\frac 2 {(t+t_0)^2}  - 3 Q^2(t+t_0) ) G, \quad t>0; \\
G(0) = 0, \, G^{\prime}(0)=1.
\end{cases}
\end{align}
Assume $t_0 \ge 0.625$. Then $G(t)>0$ and $G'(t)>0$ for all $t>0$, and 
\begin{align}
G(t) >  c \; t^{c_2}, \quad t\ge 2.5,
\end{align}
where $c>0$, $c_2>0$ are constants.
\end{prop}

\begin{proof}
By Lemma \ref{lem_Qa1}, for $t_0\ge  1.2$, we have
$\frac 2 {(t+t_0)^2}  - 3 Q^2(t+t_0)>0$ for all $t\ge 0$. In this case 
the solution obviously grows in time.
Thus it is enough to consider the case $t_0\in [0.625, 1.2]$.

By Lemma \ref{lem_Qa1}, we  have  $\frac{2}{t^2}-3 Q^2(t)\ge -6.04$ for $t\in [0.625, 1.2]$.
We compare $G$ with $H=\frac {\sin(\sqrt{6.04} t)} {\sqrt{6.04}} $ which solves
\begin{align}\notag
\begin{cases}
H^{\prime\prime} = -6.04H, \quad t>0; \\
H(0) = 0, \, H^{\prime}(0)=1.
\end{cases}
\end{align}
Note that $0.625+0.575=1.2$ and $\sqrt{6.04}\cdot 0.575<\frac {\pi}2$. Thus
$G(t)>H(t)>0$ and $G'(t)>0$ for $t\in [0, 0.575]$.  Since $t+t_0\ge 1.2$ for $t\ge 0.575$, we have
$\frac 2 {(t+t_0)^2} - 3 Q^2(t+t_0) >0$ for all $t\ge 0.575$.
 Thus $G$ and $G'$ stay positive for all $t>0$.

Next we show the asymptotic behavior of $G$.

By \eqref{4.7aa}, we have
\begin{align}
\frac 2 {(t+t_0)^2} - 3 Q^2(t+t_0) \ge \frac {1} {(t+t_0)^2}>
\frac {0.39}{(t+t_0)^2}, \quad\forall\, t\ge 2.5.
\end{align}
Consider the auxiliary system
\begin{align}
\begin{cases}
G_1^{\prime\prime} = \frac {0.39}{(t+t_0)^2}  G_1, \quad t > 2.5;\\
G_1(2.5){=G(2.5)}>0, \quad  G_1^{\prime}(2.5){= G^{\prime}(2.5)} >0.
\end{cases}
\end{align}
We decompose $G_1$ as $G_1=G_2+G_3$, where
\begin{align}
\begin{cases}
G_2^{\prime\prime}=\frac {0.39}{(t+t_0)^2} G_2, \quad t>2.5;\\
G_2(2.5)= G(2.5) - G_3(2.5), \quad G_2^{\prime}(2.5)=G^{\prime}(2.5)-G_3^{\prime}(2.5);
\end{cases}
\notag \\
\begin{cases}
G_3^{\prime\prime}= \frac {0.39}{(t+t_0)^2} G_3, \quad t>2.5;\\
G_3(2.5) = \delta_1 \cdot (2.5+t_0)^{1.3}, \quad G_3^{\prime}(2.5)= \delta_1 \cdot 0.3 \cdot (2.5+t_0)^{0.3}.
\end{cases}
\end{align}
Here $\delta_1= 0.1\cdot \min\{G(2.5), \, G^{\prime}(2.5) \} \cdot (2.5+t_0)^{-1.3}$.  Apparently
$G_2(t)>0$, $G_2^{\prime}(t)>0$ for all $t\ge 2.5$, and $G_3(t)= \delta_1 \cdot (t+t_0)^{1.3}$ for $t\ge 2.5$. 
The desired growth of $G(t)$ easily follows.
\end{proof}

\section{The case $l=0$ for the eigenvalue case of $L_+$}\label{sec_6}

We  consider the equation
\begin{align}
\Bigl( -(\partial_{tt} + \frac 2 t \partial_t) + 1-\lambda  -3 Q^2(t)  \Bigr) f =0.
\end{align}

Denote $F_{\epsilon}(t)= t f(t)$ and $\epsilon = 1-\lambda \in [0,1]$. It suffices to study the 
equation
\begin{align}
\begin{cases}
F_{\epsilon}^{\prime\prime} = (\epsilon -3 Q^2 ) F_{\epsilon}, \quad t>0;\\
F_{\epsilon}(0)=0, \; F_{\epsilon}^{\prime}(0)=-1.
\end{cases}
\end{align}

We chose the normalization $F_{\epsilon}^{\prime}(0)=-1$ since $F_{\epsilon}$ will change sign at least once.
This is proved in the following lemma.

\begin{lem} \label{lem6.1a}
Let $\epsilon \in [0,1]$. Then $F_{\epsilon} $ must change its sign at least once. The first positive zero
$t_{\epsilon}$ of $F_{\epsilon}$ satisfies 
\begin{align}
t_{\epsilon} \ge t_0>0,
\end{align}
where $t_0$ is the first positive zero of $F_0$.
\end{lem}
 
\begin{rem}
Interestingly, there exists a natural correspondence between our linearized
equation and the usual Bessel equation, at least near $t=\infty$. To see this 
consider the equation
\begin{align}
\frac {d^2}{dt^2} F_1  +(3Q^2 - \epsilon^2) F_1 =0.
\end{align}
Near $t=\infty$ one can regard $Q(t) \sim t^{-1} e^{-t}$. Dropping the $t^{-2}$ factor,
we arrive at the model
\begin{align}
\frac {d^2}{dt^2} F = (\epsilon^2- k^2 e^{-2t} ) F.
\end{align}
Make a change of variable $x=e^{-t}$.  Clearly
\begin{align}
& \frac d {dt} F = -u^{\prime} \cdot e^{-t}, \qquad (\text{here we write $F(t) = u(x)=u(e^{-t})$}),\\
& \frac {d^2}{dt^2} F = u^{\prime\prime} e^{-2t} + u^{\prime} e^{-t}
= x^2 u^{\prime\prime} +xu^{\prime}.
\end{align}
Thus we obtain
\begin{align}
x^2 u^{\prime \prime} + x u^{\prime} = (\epsilon^2 - k^2 x^2) u.
\end{align}
By another change of variable,  we arrive at the usual Bessel equation:
\begin{align}
x^2 u^{\prime\prime}+ x u^{\prime} = (\epsilon^2 - x^2)u.
\end{align}
\end{rem}
 
\begin{proof}[Proof of Lemma \ref{lem6.1a}]
We first show that $F_{\epsilon}$ must change its sign. Assume that $F_{\epsilon}$ is negative
for all $0<t<\infty$.   Denote $G_{\epsilon}=-F_{\epsilon}$. Consider 
\begin{align}
\begin{cases}
G^{\prime\prime}= (1+\gamma - 3Q^2) G, \\
G(0)=0, \, G^{\prime} (0)=1.
\end{cases}
\end{align}
Here $-\gamma<0$ corresponds to the negative eigenvalue of $L_{+}$ and $G$ is the corresponding
eigenfunction which is positive on $(0,\infty)$.  Observe that
\begin{align}
\begin{cases}
G_{\epsilon}^{\prime\prime}= (\epsilon - 3Q^2) G_{\epsilon}, \\
G_{\epsilon}(0)=0, \, G_{\epsilon}^{\prime} (0)=1.
\end{cases}
\end{align}
Since we assume $G_{\epsilon}>0$ on $(0,\infty)$, it follows by using comparison that
\begin{align}
0<G_{\epsilon}(t) \le G(t),  \qquad \forall\, 0<t<\infty.
\end{align}
Note that $G(t)$ decays as $e^{-\sqrt{1+\gamma} t}$ as $t\to \infty$.\footnote{This can be achieved via the contraction argument provided in Proposition \ref{propA1a} and Lemma \ref{lres}. Alternatively, one may apply a standard WKB-type estimate (cf. \cite{O74}).}  This clearly
contradicts the decay of $G_{\epsilon}$. Thus we arrive at a contradiction. It follows
that $F_{\epsilon}$ must change sign at least once on $(0,\infty)$. 

The proof of $t_{\epsilon}\ge t_0$ follows from comparing $F_{\epsilon}$ with $F_0$. 
\end{proof}

We now consider 
\begin{align}
\begin{cases}
F_{\epsilon}^{\prime\prime} = (\epsilon - 3Q^2(t+t_{\epsilon} ) ) F_{\epsilon}, 
\quad t>0;\\
F_{\epsilon}(0)=0, \; F_{\epsilon}^{\prime}(0)=1.
\end{cases}
\end{align}

Note that
\begin{align}
\epsilon - 3Q^2(t+t_{\epsilon} ) \ge -3Q^2(t+t_0).
\end{align}
We only need to examine the $\epsilon$-independent system
\begin{align}
\begin{cases}
q^{\prime\prime} = -3Q^2(t+t_0) q, \quad t>0;\\
q(0)=0, \; q^{\prime}(0)=1.
\end{cases}
\end{align}
\begin{thm} \label{thm6.1a}
We have $q(t)>0$ for all $0<t<\infty$. Furthermore $\min_{t\ge 1} \frac 1 t q(t) \ge c_0>0$ 
for some constant $c_0$.
\end{thm}
\begin{proof}
This follows from Proposition \ref{propC1a} in the appendix.
\end{proof}

\section{$L_-$ has no eigenvalues in $(0,1]$} \label{sec_Lminus}
The proof for $L_-$ is similar. Thus we only sketch the needed modifications. 
It suffices to examine the equation
\begin{align}\label{L-1}
(-\partial_{rr} -\frac 2 r \partial_r +\frac{l(l+1)}{r^2}+1-\lambda -Q^2) R_l=0.
\end{align}

{We first consider the case $l=0$.} Note that for $\lambda =0$, $R_0(r) =Q(r)$ is a solution to the above equation. 
 Denote $t=r$, $H_{\epsilon}(t)= t R_0(t)$ and $\epsilon = 1-\lambda \in [0,1)$. It suffices to study the 
equation
\begin{align}
\begin{cases}
H_{\epsilon}^{\prime\prime} = (\epsilon - Q^2 ) H_{\epsilon}, \quad t>0;\\
H_{\epsilon}(0)=0, \; H_{\epsilon}^{\prime}(0)=-1.
\end{cases}
\end{align}

\begin{lem} \label{lem7.1a}
Let $\epsilon \in [0,1)$. Then $H_{\epsilon} $ must change its sign at least once. The first positive zero
$\tau_{\epsilon}$ of $H_{\epsilon}$ satisfies 
\begin{align}
\tau_{\epsilon} \ge \tau_0>0,
\end{align}
where $\tau_0$ is the first positive zero of $H_0$.
\end{lem}
\begin{proof}
The proof is similar to Lemma \ref{lem6.1a}. One can use the comparison function $H_1(t)
=\mathrm{const} \cdot t Q(t)$ to deduce that $F_{\epsilon}$ for $\epsilon\in [0,1)$ must
change sign. 
\end{proof}
\begin{rem} \label{rem7.1a}
For $\epsilon =1$, the system for $H_1$ apparently has a unique solution which is 
 $- \frac 1 {Q(0)} t Q(t)$.
\end{rem}

Similar to the argument in Section 6, we only need to examine the system
\begin{align}
\begin{cases}
p^{\prime\prime} = -Q^2(t+\tau_0) p, \quad t>0;\\
p(0)=0, \; p^{\prime}(0)=1.
\end{cases}
\end{align}
\begin{thm} \label{thm7.1a}
We have $p(t)>0$ for all $0<t<\infty$. Furthermore {$\min_{t\ge 1} \frac{p(t)}{t} \ge c_1>0$ }
for some constant $c_1$.
\end{thm}
\begin{proof}
This follows from  Proposition \ref{propE1a} in the appendix.
\end{proof}
Consequently, (\ref{L-1}) does not admit any $L^2$ solution when $l=0$.

In the case $l\ge 1$, we consider the equation
\begin{align*}
\Bigl( -(\partial_{tt} + \frac 2 t \partial_t ) + 1-\lambda 
+ \frac {l(l+1)} {t^2} - Q^2(t)  \Bigr) f=0,
\end{align*}
where $\lambda \in (0,1]$. 
For $l\ge 1$, from Lemma \ref{lem_Qa1}, we have the pointwise bound
\begin{align}\label{L-2}
\frac {l(l+1) } {t^2} >Q^2(t), \qquad \forall\, t>0.
\end{align}
It follows that the above system cannot admit any nontrivial $L^2$ solution.

\appendix

\section{Some auxiliary estimates}

Consider the ODE:
\begin{align}
F^{\prime\prime}= ( \frac 2 {t^2} + V(t) ) F, \quad t>0.
\end{align}

\begin{prop} \label{propA1a}
Assume $V:\, [0, 1] \to \mathbb R$ is $C^3$ continuous and $V^{\prime}(0)=0$. 
Then any real-valued solution
 $F \in C^2((0,1))$ can be written as a linear combination of two special solutions, namely
\begin{align}
F(t) = c_1 (t^2 +\eta_1(t) ) + c_2 ( t^{-1} + \eta_2(t) ),
\end{align}
where $c_1, c_2 \in \mathbb R$,  $\eta_1, \eta_2 \in C^2([0,1])$ with  $\eta_1(t) = \mathcal O(t^4)$,
$\eta_2(t) =\mathcal O(t)$ as $t\to 0^+$.

In particular if $F\in C^2((0,1))$ and is bounded on $[0, 1]$, then 
\begin{align}
F(t) = c_1 (t^2+\eta_1(t)),
\end{align}
for some constant $c_1 \in \mathbb R$. If furthermore $\lim_{t\to 0^+} \frac{|F(t)|}{t^2}=0$, then
$F(t) \equiv 0$. 
\end{prop}
\begin{proof}
We first construct $\eta_1$ and $\eta_2$ on the interval $[0, \epsilon_0]$ where $\epsilon_0>0$ will be taken
sufficiently small. 

We write $\eta_1 = t^2 \phi_1$. The equation for $\phi_1$ reads as
\begin{align}
\phi_1^{\prime\prime} t^2 +4 t \phi_1^{\prime} =  V(t) t^2 (1+\phi_1).
\end{align}
We seek $\phi_1$ as a solution to the integral equation
\begin{align}
\phi_1(t) = \int_0^t \tau^{-4} \int_0^{\tau} V s^4 (1+\phi_1) ds d\tau.
\end{align}
Clearly the above integral equation has a unique solution in $C([0, \epsilon_0])$ for $\epsilon_0>0$ sufficiently small.
Furthermore it is clear that 
\begin{align}
|\phi_1(t) | \lesssim t^2, \qquad \text{as $t\to 0^+$}.
\end{align}

The construction of $\eta_2$ requires some more care.  Firstly we write $\eta_2= -0.5 V(0) t +ct^3+ t^2 \phi_2$ and
observe that
\begin{align}
(t^{-1} -0.5 V(0) t+c t^3 +t^2 \phi_2)^{\prime\prime}
= (2t^{-2} +V(t) ) (t^{-1} -0.5 V(0) t+ ct^3+t^2 \phi_2).
\end{align}
Simplifying a bit, we obtain
\begin{align}
(t^2 \phi_2)^{\prime\prime} =\underbrace{ t^{-1} (V(t)-V(0) )-0.5 V(0) V(t) t -4ct +c V(t) t^3}_{=:F_1} + (2t^{-2} +V) t^2 \phi_2.
\end{align}
Taking $c=\frac 14 (0.5 V^{\prime\prime}(0) -0.5 V(0)^2)$, we see that $F_1(t) = \mathcal O(t^2)$ as $t\to 0^{+}$. 
Clearly
\begin{align}
(t^4 \phi_2^{\prime})^{\prime}=F_1 t^2 +Vt^4 \phi_2.
\end{align}
Thus we seek $\phi_2$ solving the integral equation
\begin{align}
\phi_2(t) = \int_0^t \tau^{-4} \int_0^{\tau}(s^2 F_1 + s^4V\phi_2)  ds d\tau.
\end{align}
Clearly we can find a unique solution to the above equation in $C([0, \epsilon_0])$ for $\epsilon_0>0$
sufficiently small.  Furthermore, it is not difficult to check that $|\phi_2(t)| = \mathcal O(t^2)$ as
$t\to 0^+$.

It is not difficult to check that the constructed functions $\eta_1, \eta_2 \in C^2([0, \epsilon_0])$.  Furthermore, 
we can extend $\eta_1$ and $\eta_2$ to be in $C^2([0, 1])$ since the coefficient of the ODE is regular away
from $t=0$.  By checking the wronskian of $t^2+\eta_1(t)$ and $t^{-1} +\eta_2(t)$, we see that these two
solutions are linearly independent. The desired conclusion easily follows. 
\end{proof}

\begin{prop} \label{propA2b}
Consider the ODE
\begin{align}
\begin{cases}
u_{\lambda}^{\prime\prime} = (1-\lambda-3Q^2) u_{\lambda}, \\
u_{\lambda}(0)=0, \; u_{\lambda}^{\prime}(0)=1.
\end{cases}
\end{align}
There exists $-\gamma<0$ such that the following hold.
\begin{enumerate}
\item If $\lambda< -\gamma$, then $u_{\lambda}(t)>0$ for all $t>0$, and  $u_{\lambda} (t) \to \infty$ as $t\to \infty$.
\item If $\lambda = -\gamma$, then $u_{\lambda}(t)>0$ for all $t>0$, and
$u_{\lambda}(t) \to 0$ as $t\to \infty$. Furthermore the decay is exponential. 
\item If $-\gamma <\lambda \le 1$, then $u_{\lambda}(t)$ change its sign exactly once, and $u_{\lambda}(t)
\to -\infty$ as $t\to \infty$. Furthermore $|u_{\lambda}(t)|\gtrsim t^{\alpha}$ for some $\alpha>0$ as $t\to \infty$. 
\end{enumerate}
\end{prop}
\begin{proof}[Proof of Proposition \ref{propA2b}]
We proceed in three steps.

\textsf{Step 1}. Define
\begin{align}
-\gamma = \sup \{\lambda: \; \text{$u_{\lambda}(t)>0$ for all $0<t <\infty$} \}.
\end{align}
Observe that for $\lambda<1 -3Q(0)^2$, we have $u_{\lambda}(t)>0$ for all $0<t<\infty$. Also
by Lemma \ref{lem6.1a} and Proposition \ref{propC1a}, for $\lambda \in [0,1 ]$ we clearly have $u_{\lambda}(t)$ change its sign
exactly once and $u_{\lambda}(t) \to -\infty$ as $t\to \infty$.  By using comparison, we can thus rule out any positive
solution $u_{\lambda}(t)$ for $\lambda>0$ (otherwise $u_{\lambda=0}(t)$ won't be sign-changing).

 It follows that  $1-3Q(0)^2\le -\gamma  \le 0$.  By using the definition of the supreme, we can find $\lambda_j \to -\gamma$
 such that $u_{\lambda_j}(t)>0$ for all $0<t<\infty$. It follows from the continuous dependence of the ODE on
 the parameter $\lambda$ that
 $u_{-\gamma}(t) \ge 0$ for all $0<t<\infty$.  Note that $u_{-\gamma}(t)>0$ for $t>0$ sufficiently small. If
 $t_0$ is the first positive zero of $u_{-\gamma}$, then we must have $u_{-\gamma}^{\prime}(t_0)<0$ which is
 apparently impossible.   Thus $u_{-\gamma}(t)>0$ for all $0<t<\infty$.  Consequently we must have
  $-\gamma<0$. 
 
 \textsf{Step 2}. We now claim that $u_{-\gamma}(t) \to 0$ as $t\to \infty$ and the decay is exponential. Indeed if this is not the
 case, then $u_{-\gamma}(t)$ must grow exponentially to $\infty$ as $t\to \infty$. 
 In particular, one can find $T>0$ sufficiently large such that $u_{-\gamma}(T)>0$, $u_{-\gamma}^\prime(T)>0$
 and $1+\gamma -3Q^2(T)>1$.
     By using perturbation, we can then find $\delta_0>0$ sufficiently small, such that $u_{-\gamma+\delta_0}(t)>0$ 
     for all $0<t\le T$ and $1+\gamma-\delta_0-3Q^2(t)>0.5$ for all $t\ge T$. It follows that
     $u_{-\gamma+\delta_0}(t)>0$ for all $0<t<\infty$ which  is an obvious contradiction to the definition of $-\gamma$.  Thus $u_{-\gamma}(t)$ must tend to zero
exponentially in time as $t\to \infty$. 

By using comparison with $u_{-\gamma}(t)$, it is clear that $u_{\lambda}(t)>0$ for all $0<t<\infty$ when
$\lambda <-\gamma$.   We claim that $u_{\lambda}(t)$ must grow exponentially to infinity as $t\to \infty$.
Suppose this is not the case. Then $u_{\lambda}(t)$ decays as $\mathcal O(e^{-\sqrt{1-\lambda} t})$ as
$t\to \infty$. But since $u_{\lambda}(t)\ge u_{-\gamma}(t)$ and $u_{-\gamma}(t)$ decays
as $\mathcal O(e^{-\sqrt{1+\gamma} t} )$, this clearly gives a contradiction.  Thus $u_{\lambda}(t)$ must
grow exponentially to infinity as $t\to \infty$.

 \textsf{Step 3}. We now consider the case $\lambda \in (-\gamma, 0)$. 
 By the definition of $-\gamma$, we see that $u_{\lambda}(t)$ must change its sign at some finite $t_\lambda>0$.
  By using comparison with $u_{\lambda=1}(t)$, we see that $t_\lambda \ge t_0$ where $t_0$ is the first 
  positive zero of $u_{\lambda=1}$.   By an argument similar to the proof of Theorem
  \ref{thm6.1a}, we conclude $u_{\lambda}(t)\to -\infty$ as $t\to \infty$. The lower estimate
  $|u_{\lambda}(t) |\gtrsim t^{\alpha}$ follows from a simple asymptotic analysis.
\end{proof}

\subsection{Absence of embedded eigenvalues}
\begin{lem} \label{lres}
Let $A\ge 0$ and $k>0$. Suppose $F$ is a smooth function solving the linear equation
\begin{equation}\label{eqF}
F^{\prime\prime} +(k^2-\frac {A} {t^2} +3 Q^2(t) ) F=0, \qquad t>1.
\end{equation}
Then for some constants $c_1$, $c_2$ we have
\begin{align}
F(t) = c_1 (\sin(k t) +\eta_1(t)) + c_2 ( \cos(k t) + \eta_2(t) ),
\end{align}
where $\eta_i(t)$ are smooth functions satisfying
\begin{align} \label{decay}
\sup_{1\le t <\infty} ( | t^\alpha \eta_1(t) | + | t^\alpha \eta_2(t)| )<\infty
\end{align}
for any $0<\alpha<1$.
\end{lem}
\begin{proof}
It suffices to exhibit two independent solutions. We consider $\eta_1$ solving the
integral equation
\begin{align}
\eta_1(t) = \frac{1}{k}\int_t^{\infty} \sin (k(s-t) )( \frac{A}{s^2}-3 Q^2(s)) (\sin(k s) +\eta_1(s) )
ds, \qquad t\ge T_1.
\end{align}
By taking $T_1$ sufficiently large, one can obtain a contraction in the norm 
$\|t^\alpha \eta_1(t) \|_{L_t^{\infty}([T_1,\infty) }$ for any $0<\alpha<1$. Clearly, the unique solution $\eta_1$ is not $0$, and
$\sin(\epsilon t)+ \eta_1(t)$ solves the original ODE  on $(T_1,\infty)$.  Solving it backward
in time and noting that $F$ is a linear equation, we obtain a smooth solution 
defined on $[1,\infty)$. 

Similarly, we can find $\eta_2$ solving
$$
\eta_2(t) = \frac{1}{k}\int_t^{\infty} \sin (k(s-t) )( \frac{A}{s^2}-3 Q^2(s)) (\cos(k s) +\eta_2(s) )
ds, \qquad t\ge T_1.
$$
We can also obtain a solution $\cos(k t) + \eta_2(t)$ on $[1,\infty)$.

 Since the Wronskian of $\sin(k t)+ \eta_1(t)$ and $\cos(k t)+ \eta_2(t)$ is clearly
nonzero for large $t$, those two solutions are independent.
\end{proof}
\noindent {\bf Remark.} In the case $A=0$ for \eqref{eqF}, we can improve \eqref{decay} to
$$
\sup_{1\le t <\infty} ( |e^t \eta_1(t) | + | e^t \eta_2(t)| )<\infty.
$$

\begin{prop}
$L_{\pm}$ has no eigenvalues in the essential spectrum. Any $\lambda>1$ is a resonance of $L_{\pm}$.
\end{prop}
\begin{proof}
We first focus on $L_+$. $\lambda=1$ is not an eigenvalue or resonance of $L_{+}$ has been proved earlier.
For any $\lambda>1$, we shall try to find a nontrivial solution  to  the equation $L_{+}f=\lambda f$.
Passing to  spherical harmonics, it suffices  to consider
\begin{align*}
( -\partial_{rr}-\frac 2 r \partial_r +\frac {l(l+1)} {r^2} -\epsilon^2
-3Q^2) R_l=0
\end{align*}
where $R_l$ are functions of $r$ only and $\lambda-1=\epsilon^2$.  Denote $F_l=r R_l$, we get 
\begin{equation}
F_l''+(\epsilon^2-\frac{l(l+1)}{r^2}+3Q^2)F_l=0.
\end{equation}

From Lemma \ref{lres}, there exists $\eta_{l_1}, \eta_{l_2}$ such that
$$
F_l(r) = c_1 (\sin(\epsilon r) +\eta_{l_1}(r)) + c_2 ( \cos(\epsilon r) + \eta_{l_2}(r) ),
$$
where $\eta_{l_i}$ has $\mathcal O(r^{-1+})$ decay as $r\to \infty$.
Since $\eta_{l_1}$ and $\eta_{l_2}$ has $r^{-\alpha}$ decay for any $
0<\alpha<1$, $F_l(r)\in L^2([0, \infty), dr)$ if and only if $c_1=c_2=0$. For any fixed $l$, taking $c_1$ and $c_2$ properly, we can always find a solution $F_l$ which is regular near $0$. Clearly then any $\lambda>1$ is a resonance of $L_+$.

The case for $L_-$ is similar.
\end{proof}

\section{The approximate ground state $\widetilde Q$} \label{AppQ1}
As was already mentioned in the introduction, Costin, Huang and Schlag constructed in 
\cite{CHS} a remarkably accurate approximate ground state $\widetilde Q$.  We recall
in this section a few important properties of $\widetilde Q$ needed  in our analysis.

Denote for $t>0$:
\begin{align}
q_1(t)&=-\tfrac{28 t^{11}}{94431}+\tfrac{127 t^{10}}{76892}-\tfrac{604 t^9}{151861}+\tfrac{293 t^8}{59051}-\tfrac{113 t^7}{80657}+\tfrac{202 t^6}{73305}-\tfrac{3 t^5}{44981}+\tfrac{54169 t^4}{401949} \notag \\
& \qquad -\tfrac{t^3}{8885055}+\tfrac{127023 t^2}{185578}+\tfrac{21539}{93423}; \notag \\
q_2(t) &=\tfrac{t^{11}}{1183575}-\tfrac{t^{10}}{120831}+\tfrac{5 t^9}{86563}+\tfrac{2 t^8}{74523}-\tfrac{17 t^7}{36388}+\tfrac{3025 t^6}{287391}-\tfrac{5162 t^5}{329873}+\tfrac{31027 t^4}{204823}\notag \\
& \qquad -\tfrac{1415 t^3}{123249}+\tfrac{295367 t^2}{428350}-\tfrac{21 t}{15850}+\tfrac{18176}{78783}; \notag \\
f_1(t) &=\frac 1 t e^{-t}; \notag \\
p_2(t) &=q_2(t)+q_1(\frac 9{10}) -q_2(\frac 9{10})+ (q_1^{\prime}(\frac 9 {10}) - q_2^{\prime}(\frac 9{10}) )
(t-\frac 9{10}); \notag \\
g(t) & = t^{-1} \left( 2 e^t \mathrm{Ei}(-4t) -e^{-t} \mathrm{Ei}(-2t) \right),
\qquad (\text{for $x<0$,   } \mathrm{Ei}(x):=\int_{-\infty}^x \frac 1 u e^u du.)
\end{align}

Define
\begin{align}
Q_s(t)=\widetilde Q(t)=
\begin{cases}
\frac 1 {q_1(t)}, \qquad \text{if $t\le \frac 9{10}$}; \\
\frac 1 {p_2(t)}, \qquad \text{if $\frac 9 {10} <t \le \frac 5 2$};\\
\frac 1 {t} \Bigl( A_1 e^{-t} + B_1 t g(t) \Bigr), \qquad \text{if $t\ge \frac 52$},
\end{cases}
\end{align}
where
\begin{align}
A_1&=
\frac{{p_2}\left(\frac{5}{2}\right) g'\left(\frac{5}{2}\right)+g\left(\frac{5}{2}\right) {p_2}'\left(\frac{5}{2}\right)}{ {p_2^2}\left(\frac{5}{2}\right) \left( {f_1}\left(\frac{5}{2}\right) g'\left(\frac{5}{2}\right)-g\left(\frac{5}{2}\right) 
f_1^{\prime} \left(\frac{5}{2}\right)\right)}; \notag \\
B_1&=-
\frac{{p_2}\left(\frac{5}{2}\right) f^{\prime}_1\left(\frac{5}{2}\right)+f_1(\frac{5}{2}) {p_2}'\left(\frac{5}{2}\right)}{ {p_2^2}\left(\frac{5}{2}\right) \left( {f_1}\left(\frac{5}{2}\right) g'\left(\frac{5}{2}\right)-g\left(\frac{5}{2}\right) 
f_1^{\prime} \left(\frac{5}{2}\right)\right)}.
\end{align}

 Throughout this work we  reserve the notation
\begin{align}\label{eta0}
{\eta_0(t) = 7\cdot 10^{-5} \frac 1 {1+t} e^{-t}}, \qquad t\ge 0.
\end{align}
The function $\eta_0(t)$ controls the relative error between $Q$ and $\widetilde Q$.

The following proposition summarizes a few properties of $\widetilde Q$ established in \cite{CHS}. 
Denote by $C_p^2$ the space of piecewise $C^2$ functions. 

\begin{prop}[Properties of $\widetilde Q$, \cite{CHS}] \label{prop_Qs1}
The approximate ground state $\widetilde Q$ defined above satisfies the following properties:
\begin{enumerate}
\item The function $\widetilde Q \in C^1([0, \infty) \cap C_p^2 ([0, \infty) )$, and
 is decreasing for $t\in [0, \frac 52]$.
\item  We have
\begin{align}
\frac {187}{69} \cdot \frac {e^{-t} } t < \widetilde Q(t) < \frac {350} {129} \frac {e^{-t}} t, 
\qquad \text{for } r\ge \frac 52.
\end{align}
\item It holds that 
\begin{align}\label{D6}
| Q(t) -\widetilde Q(t) | \le \eta_0(t), \quad |Q^{\prime}(t) - \widetilde Q^{\prime}(t)
| \le 5 \eta_0(t), \qquad \forall\, t\ge 0.
\end{align}
\end{enumerate}
\end{prop}
\begin{proof}
The first two properties are proved in Lemma 2.4 of \cite{CHS}. The pointwise bound (3)
is derived in Lemma 3.6 of \cite{CHS}.
\end{proof}

\section{Analysis of the equation $u^{\prime\prime}=-3Q^2 u$}

Consider
\begin{align}
\begin{cases}
u^{\prime\prime}= -3 Q^2 u, \\
u(0)=0, \quad u^{\prime}(0)=-1.
\end{cases}
\end{align}

We shall prove the following.
\begin{prop} \label{propC1a}
There exists $t_* \in (0.491, 0.493)$ such that $u(t)<0$ for $0<t<t_*$, and $u(t)>0$ for all $t>t_*$.
Furthermore
\begin{align}
\lim_{t\to \infty} \frac {u(t) }{t} = \lim_{t\to \infty} u^{\prime}(t) = c_u, 
\end{align}
where $c_u>0$ is a finite constant.
\end{prop}

\begin{proof}[Proof of Proposition \ref{propC1a}]
This follows from Proposition \ref{propC1b} and Proposition \ref{propC2a}. In Proposition 
\ref{propC1b}, we examine $\tilde f(t)=-u(t)$ by using an explicit approximate solution on the interval
$[0, 0.493]$. In Proposition \ref{propC2a}, we compare $u(t)$ for $t>t_0$ with a suitable
lower solution which grows linearly in $t$ as $t\to \infty$.
\end{proof}

Define
\begin{align}
&f_{\mathrm{app}}(t)=-109t^9 + \tfrac{1885}{10}t^8 - \tfrac{5686}{100}t^7 - \tfrac{8616}{100}t^6 + \tfrac{688}{10}t^5 - \tfrac{406}{100}t^4 - \tfrac{905}{100}t^3 - \tfrac{16}{1000}t^2 + t; \notag \\
&F_{\mathrm{app}}(t)= f_{\mathrm{app}}^{\prime\prime}(t)+3 Q(t)^2 f_{\mathrm{app} }(t).
\end{align}

Let 
\begin{align}
\begin{cases}
\tilde f^{\prime\prime}(t) = -3 Q(t)^2 \tilde f(t), \qquad \\
\tilde f(0)=0, \; \tilde f^{\prime}(0)=1.
\end{cases}
\end{align}

\begin{prop} \label{propC1b}
The following hold.

\begin{enumerate}
\item $f_{\mathrm{app}}(t)>7.5\cdot 10^{-4}$ for any $0.01\le t \le 0.491$.  Furthermore $f_{\mathrm{app}}^{\prime}(t) \ge 0.996$
for any $t \in [0, 0.01]$, and
$f_{\mathrm{app}}^{\prime}(t) <-0.6$ for any $t\in [0.49, 0.5]$.
\item $f_{\mathrm{app}}(0.491) >7.5\cdot 10^{-4}$, $f_{\mathrm{app}}(0.493)<-7.2\cdot 10^{-4}$.
\item  $\int_0^{0.5} |F_{\mathrm{app}}(t)| dt \le 0.00299$.
\item $\max_{0\le t \le 0.5} | \tilde f^{\prime}(t)-f_{\mathrm{app}}^{\prime}(t)|\le 0.00299$ and
$\max_{0\le t \le 0.5} | \tilde f(t) -f_{\mathrm{app}}(t) | \le 7.1 \cdot 10^{-4}$. 
\item The first positive zero of $\tilde f$ occurs at some $t_* \in (0.491, 0.493)$. 
\end{enumerate}
\end{prop}
\begin{proof}
Item (1) and item (2) can be easily verified by using the explicit form of $f_{\mathrm{app}}$.  

For  item (3),  we first denote $$F_1=f_{\mathrm{app}}^{\prime\prime} +3\tilde Q^2 f_{\mathrm{app}}, \quad F_2=f_{\mathrm{app}} \cdot 3\cdot (2\tilde Q+\eta_0) \eta_0= 42 \cdot 10^{-5} \cdot   \tfrac{f_{\mathrm{app}}}{q_1(t)(1+t)} e^{-t}+ 3 f_{\mathrm{app}}\cdot \eta_0^2,$$
where $\eta_0= 7\cdot 10^{-5}/(1+t) e^{-t}$. Note that $|F_{\mathrm{app}}(t)|\le |F_1(t)|+|F_2(t)|$, it suffices to show
\begin{align}
\int_0^{0.5} |F_1(t)| dt \le 0.00291, \label{C6}\\
 \int_0^{0.5} |F_2(t)|dt \le 8\cdot 10^{-5}. \label{C7}
\end{align}

 Estimate of \eqref{C6}. 
Recall that $\tilde{Q}=1/q_1>0$ on $[0,0.5]$, we have
\begin{align}\label{B11}
\max_{0\le t\le 0.5}|F_1^{\prime}(t) |=\max_{0\le t\le 0.5} \left|\frac{F_1^{\prime}(t) q_1(t)^3}{q_1(t)^3}\right|<\frac{5}{2}.
\end{align}
Indeed, this follows from 
$$
\frac{5}{2}  q_1(t)^3-F_1^{\prime}(t) q_1(t)^3>0, \quad \frac{5}{2}  q_1(t)^3+F_1^{\prime}(t) q_1(t)^3>0,\; \forall t\in [0,\frac 1 2],
$$ 
where both left-hand sides are polynomials in $t$.

We take $N= 10^4$ and discretize $[0, 0.5]$ into $N$ equal subintervals with $\Delta t= \frac {1}{2 N}$.
Denote $t_i=i\Delta t$ with $0\le i\le N-1$. Clearly, the mean value theorem and \eqref{B11} yield\footnote{
We stress that the Riemann sum computed here is fully rigorous, as it amounts to adding up a finite collection of rational numbers.}
\begin{align}
\int_0^{0.5} |F_1(t)| dt & = \sum_{i=0}^{N-1} \int_{t_i}^{t_{i+1}}|F_1(t)|- |F_1(t_i)| +|F_1(t_i)|dt \notag
\\
&= \sum_{i=0}^{N-1} \int_{t_i}^{t_{i+1}}|F_1^{\prime}(\xi_i) (t-t_i)|dt+  \sum_{i=0}^{N-1} |F_1(t_i)| \Delta t \notag \\
& \le  \sum_{i=0}^{N-1} \max_{0\le t \le 0.5} |F_1^{\prime}(t) |\cdot
\frac{(\Delta t)^2}{2}   +  \sum_{i=0}^{N-1} |F_1(t_i)| \Delta t   \notag \\
&\le  0.00291. \label{rsum}
\end{align}

 Estimate of \eqref{C7}.  It is not difficult to check that 
\begin{align}
\max_{0\le t \le 0.5} |f_{\mathrm{app}}(t) | \le \frac{139}{1000},\qquad \max_{0\le t \le 0.5}  \left|\frac{f_{\mathrm{app}}}{q_1(t)(1+t)} \right|\le \frac{45}{100},
\end{align}
where the second inequality follows from the argument in \eqref{B11}.   
This clearly implies
\begin{align}
 \int_0^{0.5} |F_2| dt \le \int_0^{0.5}  42 \cdot 10^{-5} \cdot   |\tfrac{f_{\mathrm{app}}}{q_1(t)(1+t)}| e^{-t}+ 3 |f_{\mathrm{app}}|\cdot (7\cdot 10^{-5})^2\,dt\le 8\cdot 10^{-5}.
\end{align}

For item (4), we note that $\theta = f_{\mathrm{app} } - \tilde f$ solves
\begin{align}
\begin{cases}
\theta^{\prime\prime} = -3  Q^2 \theta + F_{\mathrm{app}}; \\
\theta(0)=0, \; \theta^{\prime}(0)=0.
\end{cases}
\end{align}
Clearly
\begin{align} 
\left( \frac 12 (\theta^{\prime})^2 + 3 Q^2 \frac {\theta^2} 2 
\right)^{\prime} = 3( Q^2)^{\prime} \frac {\theta^2} 2 + F_{\mathrm{app}} \theta^{\prime}.
\end{align}
Using $Q^{\prime}(t)<0$ for  $t>0$, we deduce\footnote{To avoid issues of differentiability,
one can define $\tilde \theta_{\epsilon}=\sqrt{ (\theta^{\prime})^2 +3 Q^2 \theta^2 +\epsilon}$ and observe
$\tilde \theta_{\epsilon}^{\prime} \le |F_{\mathrm{app}}| \Rightarrow \max_{0\le t\le 0.5} |\tilde \theta_{\epsilon}|
\le \int_0^{0.5}|F_{\mathrm{app}}(s)| ds$. Sending $\epsilon \to 0^+$ then yields the bound.} from item (3) that
\begin{align} \label{C6a_1}
\max_{0\le t\le 0.5} |\theta^{\prime}(t)| &\le \int_0^{0.5} |F_{\mathrm{app}}(s) | ds  \le  0.00299.
\end{align}
Also 
\begin{align}
\max_{0\le t \le 0.5} |\theta (t) | \le \frac {0.00299}{\sqrt{3} Q(0.5)} \le 7.1\cdot 10^{-4}.
\end{align}

For item (5), clearly $\tilde f^{\prime}(t) \ge 0.9$ for $0\le t \le 0.01$ by item (1) and item (4). Thus $\tilde f(t)>0$ for $t \in (0, 0.01]$. Since
$f_{\mathrm{app}}(t) >7.5\cdot 10^{-4}$ for $t\in [0.01, 0.491]$ and 
$\max_{0\le t \le 0.5}|\tilde f(t)-f_{\mathrm{app}}(t)|\le 7.1\cdot 10^{-4}$, it is clear that 
$\tilde f(t)>0$ for $ t \in [0.01, 0.491]$. 
Note that $\tilde f(0.493)<0$ by item (2) and item (4). Thus the first positive zero $t_* \in (0.491, 0.493)$.
\end{proof}

\bigskip
We shall take $t_0=0.491$ and consider the ``time-shifted" problem:
\begin{align}
\begin{cases}
\tilde f^{\prime\prime} = -3 (Q(t+t_*-t_0) )^2\tilde f, \qquad t>t_0;\\
\tilde f(t_0)=0, \;  \tilde f^{\prime}(t_0)=1;
\end{cases}
\qquad
\begin{cases}
f^{\prime\prime} = -3  Q(t)^2 f, \qquad t>t_0;\\
f(t_0)=0, \;  f^{\prime}(t_0)=1.
\end{cases}
\end{align}

Note that $\tilde f(t) =\frac 1 {u^{\prime}(t_*)} u(t+t_*-t_0)$ (clearly $u^{\prime}(t_*)>0$).  Also note
\begin{align}
Q(t+t_*-t_0) \le Q(t), \qquad \forall\, t\ge t_0.
\end{align}

By comparing $\tilde f$ with $f$, we can deduce that $\tilde f$ grows linearly in $t$ to $\infty$ as
$t\to \infty$.  This is shown in the following proposition.

\begin{prop} \label{propC2a}
We have $f(t)>0$ for all $t>t_0$, and $\lim_{t\to \infty} \frac {f(t) } {t} = c_{f} >0$,
where $c_f$ is a positive constant. Consequently,
$\tilde f(t)>0$ for all $t>t_0$, and $\lim_{t\to \infty} \frac {\tilde f(t) } {t} = c_{\tilde f} >0$,
where $c_{\tilde f}$ is a positive constant.
\end{prop}

To prove Proposition \ref{propC2a}, we first need to establish some auxiliary integral estimates.

\begin{lem} \label{lemC1a}
For $t_0=0.491$, we have
\begin{align}
\int_{t_0}^{2.2} 3 Q(t)^2 f(t) dt < 0.9794.
\end{align}
Also 
\begin{align}
0<f^{\prime}(2.2)<0.0641, \qquad 0<f(2.2)<0.542.
\end{align}
\end{lem} 
\begin{proof}
Define
\begin{align}
g(t)&=\tfrac{469 t^{10}}{25000}-\tfrac{2221 t^9}{10000}+\tfrac{1159 t^8}{1000}-\tfrac{438 t^7}{125}+\tfrac{6779 t^6}{1000}-\tfrac{1737 t^5}{200}+\tfrac{7209 t^4}{1000}-\tfrac{829 t^3}{250}+\tfrac{977 t^2}{50000}+t, \\
G(t)&=g^{\prime\prime}(t) +3Q(t+t_0)^2 g(t).
\end{align}
Note that 
\begin{align}
\begin{cases}
g^{\prime\prime}(t) + 3 Q(t+t_0)^2 g(t) = G(t); \\
g(0)=0, \; g^{\prime}(0)=1.
\end{cases}
\end{align}
We regard $g(t)$ as an approximation of $f(t+t_0)$. 

Thanks to the explicit form of $g$, it is not difficult to check that (see Remark \ref{rem_B1aux})
\begin{align} \label{IntB17}
\int_{t_0}^{2.2} 3 Q(t)^2\,  g(t-t_0)\,  dt< 0.955.
\end{align}

Next we need to estimate the difference  $g(t)-f(t+t_0)$. 
Note that (see Remark \ref{rem_B1aux})
\begin{align}
& \int_0^{2.2-t_0} |G(t) | dt=\int_{t_0}^{2.2} |G(t-t_0) | dt   \le 0.02453, \label{IntB18} \\
& (2.2-t_0) \int_0^{2.2-t_0} |G(t) | dt= (2.2-t_0) \int_{t_0}^{2.2} |G(t-t_0) | dt  \le 0.04193. \notag 
\end{align}

By a similar estimate as in \eqref{C6a_1}, we obtain
\begin{align}\label{C24}
\max_{t_0\le t \le 2.2 } |g^{\prime}(t-t_0)-f^{\prime}(t)| \le 0.02453, 
\qquad \max_{t_0 \le t \le 2.2 } |g(t-t_0) -f(t)| \le 0.04193.
\end{align}
From the explicit form of $g$, triangle's inequality yields
\begin{align}
0<f^{\prime}(2.2)<0.0641, \qquad 0<f(2.2)<0.542.
\end{align}
Similar to \eqref{C24}, we have the following refined estimate on $[t_0,1.2]$:
\begin{align} \label{IntB21}
\max_{t_0\le t\le 1.2} |g(t-t_0)-f(t) |  \le (1.2-t_0)  \int_0^{1.2-t_0} |G(t) | dt  \le 0.00183.
\end{align}
This together with \eqref{C24} yield
\begin{align}
 & \int_{t_0}^{2.2}3 Q(t)^2 |g(t-t_0)-f(t) | \, dt  \notag \\
 \le & 0.00183 \int_{t_0}^{1.2} 3Q(t)^2 dt+
 0.04193 \int_{1.2 }^{2.2} 3 Q(t)^2 dt \notag \\
 \le &  0.0244. \label{IntB22a} 
 \end{align}
The desired result follows.
\end{proof}

The next lemma gives a ``tail estimate" of $f$. Note that by the preceding lemma we have
$0<f(2.2)<0.542$ and $0<f^{\prime}(2.2)<0.0641$.

\begin{lem} \label{lemC2a}
We have
\begin{align}
\int_{2.2}^{\infty}  3  Q(t)^2 f(t)\,dt \le \int_{2.2}^{\infty}  3  Q(t)^2 \cdot ( 0.542+0.0641(t-2.2) ) \, dt < 0.02.
\end{align}
\end{lem}
\begin{proof}
From Proposition \ref{prop_Qs1}, $\tilde Q(t) < \frac {350}{129} \frac {e^{-t}} t$ for $t\ge 2.5$, we have
\begin{align}
& \int_{\frac 52}^{\infty}  3  Q(t)^2 \cdot ( 0.542+0.0641(t-2.2) ) \, dt  \notag \\
\le &\int_{\frac{5}{2}}^{\infty}    3 \left( \tfrac{350}{129}\, \tfrac{e^{-t}}{\frac{5}{2}}+7\cdot 10^{-5}\, \tfrac{ e^{-t} }{1+\frac{5}{2}}\right)^2 \cdot  ( \tfrac{542}{1000}    +\tfrac{641 }{10000}  (t-\tfrac{22}{10} ) ) \, dt <\tfrac{1}{100}.
\end{align}
On the other hand, thanks to the explicit form\footnote{In particular, 
note that for $0.9 \le t\le 2.5$, $\tilde Q$ is
a simple rational function.} of $\tilde Q$, it is not difficult to check that
\begin{align}
&\int_{2.2}^{2.5}  3  Q(t)^2 \cdot ( 0.542+0.0641(t-2.2) ) \, dt  \notag \\
\le &\int_{\frac{22}{10}}^{\frac{5}{2}}    3 (\tilde Q\left(\tfrac{22}{10} )+ 7\cdot 10^{-5}\right)^2 \cdot ( \tfrac{542}{1000}    +\tfrac{641 }{10000}  (t-\tfrac{22}{10} )) \, dt   <\tfrac{1}{100}.
\end{align}
The desired result follows easily.
\end{proof}

We are now ready to complete the proof of Proposition \ref{propC2a}.
\begin{proof}[Proof of Proposition \ref{propC2a}]
Firstly we show that $f$ stays positive on the whole interval $[t_0, \infty)$.
To prove this we argue by contradiction.
let $t_1>t_0$ be the first possible zero of $f$ on the interval $[t_0, \infty)$.
 Clearly we have $f(t)>0$ for all $t_0<t<t_1$ and $f^{\prime}(t_1)<0$. Note that
the case $f^{\prime}(t_1)=0$ is ruled out since it will imply $f(t) \equiv 0$ on the whole interval $[t_0, \infty)$. 
Since $-f^{\prime\prime} =3 Q^2 f$,  we have
\begin{align}
 \int_{t_0}^{t_1} 3 Q^2 f dt = f^{\prime}(t_0)-f^{\prime}(t_1) >1.
\end{align}

On the other hand, by Lemmas \ref{lemC1a} and \ref{lemC2a}, we have 
\begin{align}
  \int_{t_0}^{t_1} 3 Q^2 f dt<0.9794+0.02=0.9994.
\end{align}

This is clearly a contradiction. Thus we must have $f(t)>0$ for all $t>t_0$.

Now by Lemmas \ref{lemC1a} and \ref{lemC2a}, we clearly have 
\begin{align}
 \int_{t_0}^{\infty} 3 Q^2 f dt <0.9794+0.02=0.9994.
\end{align}

Thus as $T\to \infty$,  
\begin{align}
f^{\prime}(T)= 1-  \int_{t_0}^T 3 Q^2 f dt  \to c_1,
\end{align}
where $0<c_1<1$. The desired conclusion for $f$ clearly follows from L'H\^opital's rule.
By a simple comparison, we deduce the desired conclusion for $\tilde f$. 
\end{proof}

\begin{rem} \label{rem_B1aux}
In this remark we sketch the proof of  \eqref{IntB17}, \eqref{IntB18},  \eqref{IntB22a}. Below, the Riemann sum computation is fully rigorous as it sums finitely many rational numbers.

\textsf{Justification of \eqref{IntB17}}:
Denote $$h(t)= 3 (\tilde Q(t) + \eta(t))^2\,  g(t-t_0)>3 Q(t)^2 g(t-t_0), $$ where $$\eta(t):=7\cdot 10^{-5}\frac{\sum_{k=0}^{10} \tfrac{(-t)^k}{k!}}{1+t}\ge \eta_0(t),\quad \forall \, t\in [t_0,2.2].$$  
Similar to \eqref{B11}, we have
\begin{align}\label{h1a}
\begin{cases}
\max_{t_0\le t\le 0.9}|h^{\prime}(t) |=\max_{t_0\le t\le 0.9} \left|\tfrac{h^{\prime}(t) ((1+t)q_1(t))^3}{((1+t)q_1(t))^3}\right|<20, 
\\
 \max_{0.9\le t\le 2.2}|h^{\prime}(t) |=\max_{0.9\le t\le 2.2} \left|\tfrac{h^{\prime}(t) ((1+t)p_2(t))^3}{((1+t)p_2(t))^3}\right|<20.
 \end{cases}
\end{align}
We take $N= 17090$ and discretize $[t_0, 2.2]$ into $N$ equal subintervals with $\Delta t= 10^{-4}$.
Denote $t_i=t_0+i\Delta t$ with $0\le i\le N-1$. Similar to \eqref{rsum}, the mean value theorem and \eqref{h1a} yield
\begin{align}
& \int_{t_0}^{2.2} 3 Q(t)^2 g(t-t_0) \,dt  \le \int_{t_0}^{2.2} |h(t)| dt \notag 
\\
 \le   & \sum_{i=0}^{N-1}  \max_{t_0\le t \le 2.2} |h^{\prime}(t) |\cdot
\frac{(\Delta t)^2}{2}   +  \sum_{i=0}^{N-1} |h(t_i)| \Delta t   \le  0.955. \label{rsum3}
\end{align}

 \textsf{Justification of \eqref{IntB18}}:
We denote 
\begin{align*}
&G_1=g^{\prime\prime}(t-t_0) +3\tilde Q(t)^2 g(t-t_0), \\
& G_2=   3(2\tilde Q+\eta_0) \eta_0 g(t-t_0)= 42 \cdot 10^{-5} \cdot   \tfrac{\tilde{Q}(t)g(t-t_0)}{ 1+t} e^{-t}+ 3 g(t-t_0)\, \eta_0^2.
\end{align*}
Clearly, $|G(t-t_0)|\le |G_1(t)|+|G_2(t)|$. It suffices to show
\begin{align}
\int_{t_0}^{2.2} |G_1(t)| dt \le   0.02448, \label{GD6a}\\
 \int_{t_0}^{2.2} |G_2(t)|dt \le  5\cdot 10^{-5}. \label{GD7a}
\end{align}
Similar to \eqref{B11}, we have
\begin{align}\label{G1a}
\max_{t_0\le t\le 0.9}|G_1^{\prime}(t) |=\max_{t_0\le t\le 0.9} \left|\tfrac{G_1^{\prime}(t) q_1(t)^3}{q_1(t)^3}\right|<4,\quad \max_{0.9\le t\le 2.2}|G_1^{\prime}(t) |=\max_{0.9 \le t\le 2.2} \left|\tfrac{G_1^{\prime}(t) p_2(t)^3}{p_2(t)^3}\right|<4.
\end{align}
We  discretize $[t_0, 2.2]$ into  $N= 42725$  equal subintervals with $\Delta t= \tfrac{2}{5 \cdot 10^{4}}$. 
Similar to \eqref{rsum}, the mean value theorem and \eqref{G1a} yield
\begin{align}
 \int_{t_0}^{2.2} |G_1(t)| dt & \le   \sum_{i=0}^{N-1} \max_{t_0\le t \le 2.2} |G_1^{\prime}(t) |\cdot
\frac{(\Delta t)^2}{2}   +  \sum_{i=0}^{N-1} |G_1(t_0+i \Delta t)| \Delta t   \le  0.02448. \label{Grsum2}
\end{align}

It is not difficult to check that 
\begin{align}
\max_{t_0\le t \le 2.2} |g(t-t_0) | \le \frac{1}{2},\qquad \max_{t_0\le t \le 2.2}  \left|\frac{\tilde{Q}(t)g(t-t_0)}{ (1+t)} \right|\le \frac{1}{5},
\end{align}
where the second inequality follows from the argument in \eqref{B11}.   
This clearly implies
\begin{align}
  \int_{t_0}^{2.2} |G_2| dt \le  \int_{t_0}^{2.2}   42 \cdot 10^{-5} \cdot   \frac 1 5  e^{-t}+ 3 \cdot \frac{1}{2}\, (7\cdot 10^{-5})^2 \,dt \le 5\cdot 10^{-5}.
\end{align}

 \textsf{Justification of \eqref{IntB22a}}: Note that $3 Q(t)^2\le 3 (\tilde Q(t)+\eta(t))^2$, where $$\eta(t):=7\cdot 10^{-5}\frac{\sum_{k=0}^{10} \tfrac{(-t)^k}{k!}}{1+t}\ge \eta_0(t),\quad \forall \, t\in [t_0,2.2].$$ 
 We discretize $[t_0, 1.2]$ into $N_1= 1418$ equal subintervals and  $[1.2, 2.2]$ into $N_2=2000$ equal subintervals with $\Delta t= 5 \cdot 10^{-4}$.  Since $3 (\tilde Q(t)+\eta(t))^2$  is decreasing in $t$ on $[t_0, 2.2]$, we have
 \begin{align*}
 \int_{t_0}^{1.2} 3 (\tilde Q(t)+\eta(t))^2 \,dt\le \sum_{i=0}^{N_1-1} 3 (\tilde Q(t_0+i \Delta t)+\eta(t_0+i \Delta t))^2\, \Delta t<\tfrac{45596}{10000}; \\
 \int_{1.2}^{2.2} 3 (\tilde Q(t)+\eta(t))^2 \,dt\le \sum_{k=0}^{N_2-1} 3 (\tilde Q(\tfrac{12}{10}+k \Delta t)+\eta(\tfrac{12}{10}+k \Delta t))^2\, \Delta t<\tfrac{3827}{10000}.
 \end{align*}
 Thus \eqref{IntB22a} follows.

\end{rem}

\section{Analysis of the equation $v^{\prime\prime}=-Q^2 v$}

Consider the ODE
\begin{align}
\begin{cases}
v^{\prime\prime} = -Q^2 v, \qquad t>0; \\
v(0)=0, \quad v^{\prime}(0) =-1.
\end{cases}
\end{align}

\begin{prop} \label{propE1a}
There exists $t_* >0$ such that $v(t)<0$ for $0<t<t_*$, and $v(t)>0$ for all $t>t_*$.
Furthermore
\begin{align}
\lim_{t\to \infty} \frac {v(t) }{t} = \lim_{t\to \infty} v^{\prime}(t) = c_v, 
\end{align}
where $c_v>0$ is a finite constant.
\end{prop}
\begin{proof}[Proof of Proposition \ref{propE1a}]
This follows from Lemma \ref{lemE1b} and Lemma \ref{lemE2b}. In particular Lemma \ref{lemE1b}
gives the control of $v$ for $t\in [0, 1.2]$.  By using a suitable comparison, Lemma \ref{lemE2b} establishes the desired asymptotic behavior  of $v$ for $t>1.2$.
\end{proof}

\begin{lem} \label{lemE1b}
For some $t_0> 1.2$, the following hold.

We have $v(t)<0$ for all $0<t<t_0$ and $v(t_0)=0$, $v^{\prime}(t_0)>0$.
\end{lem}
\begin{proof}
We only need to show that the first positive zero of $v$ occurs at some $t_0$ with $t_0\ge 1.2$.
We define the following polynomial ansatz
\begin{align}
&w(t)=\tfrac{321 t^{10}}{250}-\tfrac{4649 t^9}{500}+\tfrac{144 t^8}{5}-\tfrac{1229 t^7}{25}+\tfrac{4887 t^6}{100}-\tfrac{1299 t^5}{50}+\tfrac{359 t^4}{100}+\tfrac{271 t^3}{100}+\tfrac{7 t^2}{200}-t.
\end{align}
Let $F(t)= w^{\prime\prime} (t)+ Q^2(t) w(t)$, we shall first prove that
\begin{align}
\int_0^{1.2}|F(t)| dt & \le 0.0358.
\end{align}
 We denote $$F_1=w^{\prime\prime} +3\tilde Q^2 w, \quad F_2=w    (2\tilde Q+\eta_0) \eta_0= 14 \cdot 10^{-5} \cdot   \tfrac{\tilde{Q}(t)w(t)}{ (1+t)} e^{-t}+  w\, \eta_0^2,$$
where $\eta_0= 7\cdot 10^{-5}/(1+t) e^{-t}$. Note that $|F(t)|\le |F_1(t)|+|F_2(t)|$, it suffices to show
\begin{align}
\int_0^{1.2} |F_1(t)| dt \le 0.03573, \label{D6a}\\
 \int_0^{1.2} |F_2(t)|dt \le 7\cdot 10^{-5}. \label{D7a}
\end{align}

Similar to \eqref{B11}, we have
\begin{align}\label{F1a}
\max_{0\le t\le 0.9}|F_1^{\prime}(t) |=\max_{0\le t\le 0.9} \left|\tfrac{F_1^{\prime}(t) q_1(t)^3}{q_1(t)^3}\right|<3,\quad \max_{0.9\le t\le 1.2}|F_1^{\prime}(t) |=\max_{0.9 \le t\le 1.2} \left|\tfrac{F_1^{\prime}(t) p_2(t)^3}{p_2(t)^3}\right|<3.
\end{align}
We  discretize $[0, 1.2]$ into $N= 4800$ equal subintervals with $\Delta t= \frac {12}{10 N}$.
 Similar to \eqref{rsum}, the mean value theorem and \eqref{F1a} yield
\begin{align}
 \int_0^{1.2} |F_1(t)| dt & \le   \sum_{i=0}^{N-1} \max_{0\le t \le 1.2} |F_1^{\prime}(t) |\cdot
\frac{(\Delta t)^2}{2}   +  \sum_{i=0}^{N-1} |F_1(i \Delta t)| \Delta t   \le  0.03573. \label{rsum2}
\end{align}

It is not difficult to check that 
\begin{align}
\max_{0\le t \le 1.2} |w(t) | \le \frac{266}{1000},\qquad \max_{0\le t \le 1.2}  \left|\frac{\tilde{Q}(t)w(t)}{ 1+t} \right|\le \frac{5}{8},
\end{align}
where the second inequality follows from the argument in \eqref{B11}.   
This clearly implies
\begin{align}
  \int_0^{1.2} |F_2| dt \le  \int_0^{1.2}  14 \cdot 10^{-5} \cdot   |\tfrac{\tilde{Q}(t)w(t)}{ (1+t)}|\, e^{-t}+   |w(t)|\cdot (7\cdot 10^{-5})^2\,dt \le 7\cdot 10^{-5}.
\end{align}

By a similar analysis as in \eqref{C6a_1}, we obtain
\begin{align}
\max_{0\le t \le 1.2}
| v^{\prime}(t) -w^{\prime}(t) |\le \int_0^{1.2} |F(t)|dt\le  0.0358. 
\end{align}
In addition, the Fundamental Theorem of Calculus yields that
$$ \max_{0\le t \le 1.2} | v(t)-w(t)|\le 1.2 \max_{0\le t \le 1.2}
| v^{\prime}(t) -w^{\prime}(t) | < 0.043.$$
On the other hand, it is not difficult to check that 
\begin{align}
\max_{0.05\le t \le 1.2} w(t)  \le -0.049, \qquad \max_{0\le t \le 0.05} w^{\prime}(t)  \le -0.9.
\end{align}
Recall that $v(0)=w(0)=0$ and $v^{\prime}(0)=w^{\prime}(0)=-1$. Clearly these imply that $v(t)<0$ for any $0<t\le 1.2$. Thus the first positive zero of $v$ must
occur at some $t_0>1.2$.
\end{proof}

\begin{lem} \label{lemE2b}
Let $\tau\ge 1.2$.
Consider the ODE
\begin{align}
\begin{cases}
y^{\prime\prime}(t) + Q(t+\tau)^2 y (t)=0, \qquad t>0; \\
y(0)=0, \quad y^{\prime}(0)=1.
\end{cases}
\end{align}
We have $y(t)>0$ for all $t>0$, and 
\begin{align}
\lim_{t\to \infty} \frac {y(t)} t = \lim_{t\to \infty} y^{\prime}(t) = b_1,
\end{align}
where $b_1>0$ is a constant.
\end{lem}
\begin{proof}
Firstly it is not difficult to check that for any $t\ge 0$, $\tau \ge 1.2$, we have
\begin{align}
Q(t+\tau)^2 \le (\tilde Q(t+1.2) + \eta_0(t+1.2) )^2 \le  e^{-0.6 (t+1.2)}.
\end{align}
Then we only need to examine the lower solution $w$ solving
\begin{align}
\begin{cases}
w^{\prime\prime} =-e^{-0.6 (t+1.2)} w,  \qquad t>0; \\
w(0)=0, \quad w^{\prime}(0)=1.
\end{cases}
\end{align}
The solution $w$ has an explicit representation in terms of Bessel functions, namely
\begin{align}
w(t)=\frac{e^{9/25} \left(Y_0\left(\frac{10}{3 e^{9/25}}\right) J_0\left(\frac{10 \sqrt{e^{-\frac{3 t}{5}}}}{3 e^{9/25}}\right)-J_0\left(\frac{10}{3 e^{9/25}}\right) Y_0\left(\frac{10 \sqrt{e^{-\frac{3 t}{5}}}}{3 e^{9/25}}\right)\right)}{J_1\left(\frac{10}{3 e^{9/25}}\right) Y_0\left(\frac{10}{3 e^{9/25}}\right)-J_0\left(\frac{10}{3 e^{9/25}}\right) Y_1\left(\frac{10}{3 e^{9/25}}\right)}.
\end{align}
Clearly 
\begin{align}
J_1\left(\frac{10}{3 e^{9/25}}\right) Y_0\left(\frac{10}{3 e^{9/25}}\right)-J_0\left(\frac{10}{3 e^{9/25}}\right) Y_1\left(\frac{10}{3 e^{9/25}}\right) 
=\frac{2}{\pi} \cdot \frac 1 {\frac{10}{3 e^{9/25}} }>0.
\end{align}
It is not difficult to check that
\begin{align}
Y_0\left(\frac{10}{3 e^{9/25}}\right) J_0\left(\frac{10 s}{3 e^{9/25}}\right)-J_0\left(\frac{10}{3 e^{9/25}}\right) Y_0\left(\frac{10 s}{3 e^{9/25}}\right)>0, \qquad\forall\, 0< s <1.
\end{align}
Furthermore $Y_0(s)/\log s \to \frac 2 {\pi}$ as $s\to 0^{+}$. 
Thus $w(t)>0$ for all $t>0$ and $\lim_{t\to \infty} \frac {w(t)} t =c$ for some positive constant $c>0$. 
By using comparison it is not difficult to check that the desired conclusion holds for the solution $y$.
\end{proof}

\end{document}